\documentclass [12pt,a4paper,reqno]{amsart}
\usepackage{color}
\usepackage{array}
\usepackage{tabularx}
\usepackage{enumerate}  
\usepackage{amsmath}
\usepackage{amsfonts}
\usepackage{amssymb}
\usepackage{bbm}
\usepackage[bbgreekl]{mathbbol}
\usepackage{mathrsfs}
\usepackage{latexsym}
\usepackage{epsfig}
\usepackage{graphicx}
\usepackage{caption}
\usepackage[colorlinks=true,linkcolor=blue,citecolor=blue]{hyperref}

\usepackage[titletoc]{appendix}

\AtBeginDocument{}

\numberwithin{equation}{section}
        
\newtheorem{definition}{Definition}[section]
\newtheorem{lemma}[definition]{Lemma}
\newtheorem{theorem}[definition]{Theorem}
\newtheorem{proposition}[definition]{Proposition}
\newtheorem{corollary}[definition]{Corollary}
\newtheorem{example}[definition]{Example}
\newtheorem{remark}[definition]{Remark}

\newenvironment{assumption}[1][Assumption]{\itshape\begin{trivlist}
\item[\hskip \labelsep {\bfseries #1}]}{\end{trivlist}}

\title{Local H\"older regularity for set-indexed processes}

\date{\today}

\author{Erick Herbin}
\address{Ecole Centrale Paris\\
MAS Laboratory and INRIA Regularity team\\ 
Grande Voie des Vignes, 92295~Chatenay-Malabry, France} \email[Corresponding author]{erick.herbin@ecp.fr}

\author{Alexandre Richard}
\email{alexandre.richard@inria.fr}

\begin{document}

\newcommand{\fin}{$\Box$\\}
\newcommand{\C}{\mathbf{C}}
\newcommand{\R}{\mathbf{R}}
\newcommand{\Q}{\mathbf{Q}}
\newcommand{\N}{\mathbf{N}}
\newcommand{\PP}{\mathbf{P}}
\newcommand{\EE}{\mathbf{E}}
\newcommand{\Acal}{\mathcal{A}}
\newcommand{\Ccal}{\mathcal{C}}
\newcommand{\ds}{\displaystyle}
\newcommand{\saut}[1]{\hfill\\[#1]}
\newcommand{\vsp}{\vspace{.15cm}}
\newcommand{\difrac}{\displaystyle \frac}
\newcommand{\dist}{\textrm{dist}}
\newcommand{\mbf}{\textbf}

\newcommand{\alphar}{\texttt{\large $\boldsymbol{\alpha}$}}
\newcommand{\qA}{q_{\underline{\mathcal{A}}}}
\newcommand{\dA}{d_{\mathcal{A}}}
\newcommand{\HA}{$\mathcal{H}_{\underline{\mathcal{A}}}$}
\newcommand{\sifbm}{\mathbf{B}}

\maketitle

\begin{abstract}
In this paper, we study the H\"older regularity of set-indexed stochastic processes defined in the framework of Ivanoff-Merzbach. The first key result is a Kolmogorov-like H\"older-continuity Theorem, whose novelty is illustrated on an example which could not have been treated with anterior tools. Increments for set-indexed processes are usually not simply written as $X_U-X_V$, hence we considered different notions of H\"older-continuity.
Then, the localization of these properties leads to various definitions of H\"older exponents, which we compare to one another. 

In the case of Gaussian processes, almost sure values are proved for these exponents, uniformly along the sample paths.
As an application, the local regularity of the set-indexed fractional Brownian motion is proved to be equal to the Hurst parameter uniformly, with probability one.
\end{abstract}



\section{Introduction}

Sample path properties of stochastic processes have been deeply studied for a long time, starting with the works of Kolmogorov, L\'evy and others on the modulus of continuity and laws of the iterated logarithm of the Brownian motion. Since the late 1960s, these results were extended to general Gaussian processes, while a finer study of the local properties of these sample paths was carried out (we refer to Berman \cite{Berman(1970),Berman(1972)}, Dudley \cite{dudley}, Orey and Pruit \cite{Orey.Pruitt(1973)}, Orey and Taylor \cite{Orey.Taylor(1974)} and Strassen \cite{Strassen(1964)}, for the early study of Gaussian paths and their rare events). 
Among the large literature dealing with fine analysis of regularity, H\"older exponents continue to be widely used as a local measure of oscillations (see  \cite{BH2, Baraka.Mountford.ea(2009), lawler2011, LLS06, tudorxiao07} for examples of recent works in this area).
Two different definitions, called local and pointwise H\"older exponents, are usually considered for a stochastic process $\{X_t;\;t\in\R_+\}$, depending whether the increment $X_t-X_s$ is compared with a power $|t-s|^{\alpha}$ or $\rho^{\alpha}$ inside a ball $B(t_0,\rho)$ when $\rho\rightarrow 0$.
As an example, with probability one, the local regularity of fractional Brownian motion $\{B^H;\;t\in\R_+\}$ is constant along the path: the pointwise and local H\"older exponents at any $t\in\R_+$ are equal to the self-similarity index $H\in (0,1)$ (e.g. see \cite{ehjlv}).

This field of research is also very active in the multiparameter context and a non-exhaustive list of authors and recent works in this area includes Ayache \cite{Ayache.Shieh.ea(2011)}, Dalang \cite{Dalang.Nualart.ea(2011)}, Khoshnevisan \cite{Dalang.Nualart.ea(2011), KX05}, L\'evy-V\'ehel \cite{ehjlv}, Xiao \cite{mwx,xiao2009,xiao2010}. 
As an extension to the multiparameter one, the set-indexed context appeared to be the natural framework to describe invariance principles studying convergence of empirical processes (e.g. see \cite{Bierme,Ossiander.Pyke(1985)}). 
The understanding of set-indexed processes and particularly their regularity is a more complex issue than on points of $\R^N$. The simple continuity property is closely related to the nature of the indexing collection. As an example, Brownian motion indexed by the lower layers of $[0,1]^2$ (i.e. the subsets $A\subseteq [0,1]^2$ such that $[0,t]\subseteq A$ for all $t\in A$) is discontinuous with probability one \cite{adlerTaylor,Ivanoff}. 
As a matter of fact, necessary and sufficient conditions for the sample path continuity were investigated, starting with Dudley \cite{dudley} who introduced a sufficient condition on the metric entropy of the indexing set, followed by Fernique \cite{Fernique} who gave a necessary conditions in the specific case of stationary processes on $\R^N$. Talagrand gave a definitive answer in terms of majorizing measures \cite{Talagrand} (see \cite{adlerTaylor} or \cite{davar} for a comprehensive survey and also \cite{alexander84,alexander} for a LIL and L\'evy's continuity moduli for set-indexed Brownian motion). We must warn the reader though, that even if we shall use metric entropy concepts, our goal here will not be to establish sufficient conditions for H\"older continuity. Instead, it is to provide a general framework, rich enough to study different types of H\"older continuity and processes whose regularity might change at each point of the sample paths. Thus, our point of view is to formulate simple hypotheses in terms of the law of the increments of the processes we consider, rather than more refined assumptions as in the previously mentioned papers.

A formal set-indexed setting has been introduced by Ivanoff and Merzbach in order to study standard properties of stochastic processes, such as martingale and Markov properties (see \cite{Ivanoff}). In this framework, an {\em indexing collection} $\Acal$ is a collection of subsets of a measure space $(\mathcal{T},m)$, which is assumed to satisfy certain properties such as stability by intersection of its elements. 
Section \ref{sec:Ho-cont} of the present paper uses these properties, with an entropic condition (\HA) on $\Acal$, to derive a Kolmogorov-like criterion for H\"older-continuity of a set-indexed process.
The collection of sets $\Acal$ is endowed with a metric $\dA$ and a nested sequence $\underline{\Acal}=(\Acal_n)_{n\in\N}$ of finite subcollections of $\Acal$ such that each element of $\Acal$ can be approximated as the decreasing limit (for the inclusion) of its projections on the $\Acal_n$'s.
We consider Assumption (\HA) on $\underline{\Acal}$ and $\dA$ which imposes that: 1) the distance from any $U\in\Acal$ to $\Acal_n$ can be related to the cardinal $k_n=\#\Acal_n$, roughly by $\dA(U,\Acal_n) = O(k_{n}^{-1/\qA})$, where $\qA$ is called the {\em discretization exponent} of $(\Acal_n)_{n\in\N}$; 
and 2) a minimality condition on the class $(\Acal_n)_{n\in \N}$ that turns to be important when processes are not Gaussian. This condition is discussed and compared to other entropic conditions, and the example of a stable-like process to which no previous result seems to be easily applicable, comes to illustrate our point.

Alternatively, H\"older-continuity can be based on the usual definition for increments of set-indexed processes. 
Instead of quantities $X_U - X_V$, the increments of a set-indexed process $\{X_U;\;U\in\Acal\}$ are defined on the class $\Ccal$ of sets $C=U_0\setminus\bigcup_{1\leq k\leq n}U_k$ where $U_0,U_1,\dots,U_n\in\Acal$ by the {\em inclusion-exclusion formula} 
\begin{align*}
\Delta X_C = X_{U_0} - \sum_{k=1}^n \sum_{j_1<\dots<j_k} (-1)^{k-1} X_{U_0 \cap U_{j_1} \cap \dots \cap U_{j_k}}.
\end{align*}
This definition extends the notion of rectangular increments for multiparameter processes. For instance, quantities like $\Delta_{[\mathbf{u},\mathbf{v}]} \mathbf{B} = \mathbf{B}_{\mathbf{v}} - \mathbf{B}_{(u_1,v_2)} - \mathbf{B}_{(v_1,u_2)} + \mathbf{B}_{\mathbf{u}}$, where $\mathbf{u}\preccurlyeq \mathbf{v}\in \R_+^2$ and $\mathbf{B}$ is the Brownian sheet, were proved to be useful to derive geometric sample path properties of the process (see e.g. the works of Dalang and Walsh \cite{DalangWalsh}). Let us notice that some processes can satisfy an increment stationarity property with respect to these increments while they do not for quantities $X_U-X_V$. Moreover, this inclusion-exclusion principle is very useful when it comes to martingale and Markov properties.
According to this definition, another way to express the H\"older-continuity of $X$ is $|\Delta X_C|\leq L\ m(C)^{\gamma}$, for $C\in\Ccal$. This question is clarified in Section \ref{sec:Kolm}.

\medskip

The purpose of H\"older exponents is the localization of the \emph{exact} H\"older-continuity concept. Following the previous discussion, the first definition for local and pointwise H\"older exponents is based on the comparison between $|X_U-X_V|$ and a power $\dA(U,V)^{\alpha}$ or $\rho^{\alpha}$ in a ball $B_{\dA}(U_0,\rho)$ around $U_0\in\Acal$ when $\rho\rightarrow 0$.
Another definition compares $|\Delta X_C|$ for $C=U\setminus\bigcup_{1\leq k\leq n}V_k$ in $\Ccal$ with $\dA(U,U_0)<\rho$ and $\dA(U_0,V_k)<\rho$ for each $k$, to a power $m(C)^{\alpha}$ when $\rho\rightarrow 0$.
These two kinds of exponents, precisely defined in Section \ref{sec:HoExp}, provide a fine knowledge of the local behaviour of the sample paths. In Section \ref{sec:Hopc}, we define the H\"older exponent of pointwise continuity of a process $X$ by comparing $\Delta X_{C_n(t)}$ with a power $m(C_n(t))^{\alpha}$ when $n\rightarrow\infty$, where $(C_n(t))_{n\in\N}$ is a decreasing sequence of elements in $\mathcal{C}$ which converges to $t\in\mathcal{T}$.
This new exponent is related to the concept of {\em pointwise continuity}, which has been introduced in the multiparameter setting in \cite{AMSW} and in the set-indexed setting in \cite{ehem3}. This is a weak form of continuity, since without any supplementary condition on the indexing collection, the set-indexed Brownian motion satisfies this property, even on lower layers where it is not continuous.

Then in Section \ref{sec:flow}, the different H\"older exponents are linked to the H\"older regularity of projections of the set-indexed process on increasing paths.

\medskip

In the Gaussian case, we prove in Section \ref{sec:as-exp} that the different aforementioned H\"older exponents admit almost sure values. 
Assumption (\HA) turns to be very well-adapted to extend this result from the multiparameter to the set-indexed setting.
Moreover, these almost sure values can be obtained uniformly on $\Acal$ for the local exponent when (\HA) holds.
However, only an a.s. lower bound can be obtained for the pointwise exponent (even for multiparameter processes). Nevertheless, we proved that an equality holds for the set-indexed fractional Brownian motion (defined in \cite{ehem}) in Section~\ref{sec:appl}, thus improving on a result in the multiparameter case \cite{ehjlv}. As this requires some specific extra work, we believe that the uniform almost sure result might not be true for the pointwise exponent of any Gaussian process. Besides, we exhibit in a subsequent article \cite{AR13} a process whose H\"older continuity changes along the sample paths, thus providing a set-indexed process with a ``non-trivial'' behaviour.

\section{H\"older continuity of a set-indexed process}\label{sec:Ho-cont}

\subsection{Indexing collection for set-indexed processes}

A general framework was introduced by Ivanoff and Merzbach to study martingale and Markov properties of set-indexed processes (we refer to \cite{Ivanoff} for the details of the theory). The structure of these indexing collections allowed the study of the set-indexed extension of fractional Brownian motion \cite{ehem}, its increment stationarity property \cite{ehem2} and a complete characterization of the class of set-indexed L\'evy processes \cite{ehem3}.

Let $\mathcal{T}$  be a  locally compact complete separable metric and measure space, with metric $d$ and Radon measure $m$ defined on the Borel sets of $\mathcal{T}$.
All stochastic processes will be indexed by a class $\mathcal{A}$ of compact connected subsets of $\mathcal{T}$.

In the whole paper, the class of finite unions of sets in any collection $\mathcal{D}$ will be denoted by $\mathcal{D}(u)$.  In the terminology of
\cite{Ivanoff}, we assume that the {\it indexing collection} $\mathcal{A}$ satisfies stability and separability conditions in the sense of Ivanoff and Merzbach:

\begin{definition}[adapted from \cite{Ivanoff}]\label{basic}
A nonempty class $\mathcal{A}$ of compact, connected subsets of $\mathcal{T}$ is 
called an {\em indexing collection}\index{indexing collection} 
if it satisfies the following:
\begin{enumerate}
\item $\emptyset \in \mathcal{A}$, $\mathcal{A}$ is closed under arbitrary intersections and if 
$A,B\in \mathcal{A}$ are nonempty, then $A\cap B$ is nonempty. 

\item {\em Separability from above}: 
There exists an increasing sequence of finite subclasses 
$\mathcal{A} _n=\{\emptyset,A_1^n,...,A_{k_n-1}^n\}$ ($n\in\N, k_n\geq 1$) of $\mathcal{A}$ closed under intersections and a
sequence of functions $g_n: \mathcal{A}\rightarrow \mathcal{A}_n$ defined by
\begin{equation*}
\forall U\in\Acal,\quad
g_n(U) = \bigcap_{V\in \Acal_n \atop V^\circ \supseteq U} V
\end{equation*}
and such that for each $U\in\mathcal{A}$, $U=\displaystyle\bigcap_{n\in\N} g_n(U)~ $.
\end{enumerate}
 (Note: `$\overline{(\cdot)}$' and`$(\cdot )^{\circ}$' denote
respectively the closure and the interior of a set.)

\end{definition}

Standard examples of indexing collections can be mentioned, such as rectangles $[0,t]$ of $\R^N$, arcs of the circle $\mathbbm{S}^2$ or lower layers. Some of them are detailed in Examples~\ref{ex:index} and \ref{ex:layers} below.

\medskip

\noindent{\bf Distances on sets.}
In order to study the H\"older-continuity of set-indexed processes, we consider a distance on the indexing collection. Along this paper, we may sometimes specify the distance on $\Acal$ that we are using. 
Among them, the following distances are of special interest: 
\begin{itemize}
\item The classical Hausdorff metric $d_H$ defined on 
$\mathcal{K}\setminus\emptyset$, the nonempty compact subsets of $\mathcal{T}$,
by
\begin{equation*}
\forall U,V\in\mathcal{K}\setminus\emptyset;\quad
d_H(U,V)=\inf\left\{\epsilon>0:U\subseteq V^{\epsilon}\textrm{ and }
V\subseteq U^{\epsilon} \right\},
\end{equation*} 
where $U^{\epsilon} = \{x\in \mathcal{T}: d(x,U) \leq \epsilon\}$; 
\item and the pseudo-distance $d_{m}$ defined by 
\begin{equation*}
\forall U, V\in\mathcal{A};\quad
d_{m}(U,V)=m\left(U\bigtriangleup V\right),
\end{equation*}
where $m$ is the measure on $\mathcal{T}$ and $\bigtriangleup$ denotes the symetric difference of sets.
\end{itemize}


\begin{remark}
In the case of $\mathcal{A}=\left\{ [0,t];\;t\in\R^N_{+} \right\}$,
$(s,t)\mapsto d_{m}([0,s],[0,t])$ induces a distance on 
$\R^N_{+}$. If $m$ is the Lebesgue measure of $\R^N$, this distance is equivalent to the euclidean distance on compacts of $\R^N$ away from the axes (\cite{hausdorff}).
\end{remark}



\medskip

\noindent {\bf Controlling the size of indexing collections.}
We recall that a metric space $(\mathcal{T},d)$ is \emph{totally bounded} if for any $\epsilon>0$, $\mathcal{T}$ can be covered by a finite number of balls of radius smaller than $\epsilon$. The minimal number of such balls is called the \emph{metric entropy} and is denoted by $N(\mathcal{T},d,\epsilon)$. The sample path properties of stochastic processes indexed on general sets is closely related to the metric entropy they induce on the indexing set (cf \cite{adlerTaylor,LedouxTalagrand}). Without further assumptions than those of Definition \ref{basic}, $\Acal$ is not totally bounded.

Notice that the sequence $(k_n)_{n\in \N} = (\# \Acal_n)_{n\in \N}$ is an increasing sequence that tends to $\infty$, as $n\rightarrow\infty$. This property comes from condition \emph{(2)} in Definition \ref{basic}. 
We will say that $(\Acal_n,k_n)_{n\in \N}$ is \emph{admissible} if:
\begin{equation}\label{admissibility}
\forall \delta>0, \quad \sum_{n=1}^{\infty} \frac{k_{n+1}}{k_n^{1+\delta}} < \infty .
\end{equation}
This should not appear as a restriction, since if $(k_n)_{n\in \N}$ was going to $\infty$ too slowly, it would suffice to extract a subsequence ; and in the opposite situation, the gap between one scale to the other is too large and can then be filled with additional subclasses.

\begin{assumption} {\bf (\HA).}
\label{AssumptDiscr2}
Let $\dA$ be a (pseudo-)distance on $\mathcal{A}$. Let us suppose that for $\underline{\Acal}=(\mathcal{A}_n)_{n\in \N}$, there exist positive real numbers $\qA$ and $M_1$ such that:
\begin{enumerate}
\item For all $n\in\N$,
\begin{equation}\label{eq:assump1}
\sup_{U\in \Acal_n} \dA(U,g_n(U)) \leq M_1\ k_{n}^{-1/\qA}, \tag{H1}
\end{equation}
\item
the collection $(\Acal_n)_{n\in\N}$ is minimal in the sense that: setting for all $n\in\N$ and all $U\in\Acal_n$,
$$\mathcal{V}_n(U) = \{V\in \Acal_n: {V\supsetneq U}, \dA(U,V)\leq 3 M_1 k_n^{-1/\qA}\},$$ 
the sequence $(N_n)_{n\geq 1}$ defined by $N_n = \max_{U\in\Acal_n} \# \mathcal{V}_n(U)$ for all $n\geq 1$ satisfies
\begin{equation}\label{eq:assump2}
\forall\delta>0,\quad
\sum_{n=1}^{\infty} k_n^{-\delta} \ N_n  <\infty. \tag{H2}
\end{equation}
\end{enumerate}
\end{assumption}

The real $\qA$ is not unique and it depends {\it a priori} on the distance $\dA$ and the sub-semilattices $\underline{\Acal}=(\Acal_n)_{n\in\N}$. Such a real $\qA$ is called {\em discretization exponent} of $(\mathcal{A}_n)_{n\in\N}$. Note that if $N_n$ can be bounded independently of $n$, then the last assumption is satisfied by admissibility of $k_n$.

\begin{remark}
\begin{itemize}
\item Without loss of generality, the distance $\dA$ can be normalized such that $M_1=1$.
\item The summability condition \eqref{eq:assump2} of Assumption (\HA) is close to the notion of entropy with inclusion developped by Dudley \cite{DudleyInclusion} in the context of empirical processes. On the contrary to the present work, \cite{DudleyInclusion} focused exclusively on the sample path boundedness and continuity in the Brownian case.
\item In \cite{XiaoHausdorff}, the same idea had already appeared in a slightly different form, reinforcing the legitimacy of this additional hypothesis for processes with heavy tails.
\item Assumption (\HA) implies the total boundedness of $(\Acal, \dA)$. Indeed, (\ref{eq:assump1}) means that for any $n\in \N^*$,
$\mathcal{A}_n$ forms a $k_n^{-1/\qA}$-net, and thus $(\mathcal{A}, \dA)$ is totally bounded.
\end{itemize}
\end{remark}

\medskip

\noindent {\bf Examples.} We present three classes of examples: the first which corresponds to the Euclidean setting and affiliated simple indexing collections, the second one which fails to satisfy Assumption (\HA), and the last which is significantly different from the Euclidean one since it has a non-integer $\qA$.

\begin{example}\label{ex:index}
In the case of $\mathcal{A}=\left\{ [0,t];\;t\in [0,1]\subset\R^N_{+} \right\}$,
the subclasses $\mathcal{A}_n$ ($n\in\N$) can be chosen as
\[ 
\left\{ [0,2^{-n}.(l_1,\dots,l_N)];\;0\leq l_1,\dots,l_N\leq 2^n \right\}.
\]
\noindent Let $U$ be a set in $\Acal$. The distance $d_\lambda$ (induced by the Lebesgue measure $\lambda$) between $U$ and $g_n(U)$ is the volume difference between the two sets. For $n$ large, it is bounded as follows:
\begin{equation*}
\sup_{U\in\Acal} d_{\lambda}(U,g_n(U)) = \sup_{U\in\Acal}  \lambda(g_n(U) \setminus U) = N. 2^{-n} + o(2^{-n}).
\end{equation*}
Since $k_n = (2^n + 1)^N$, this leads to $d_{\lambda}(U,g_n(U)) = O(k_{n}^{-1/\qA})$, where $\qA = N$, and the other condition of Assumption (\HA) are satisfied.
\end{example}

The following example shows that the collection of {\em lower layers} of $\R^N$ does not satisfy Assumption (\HA). We will see later that this result is not surprising in the view of Theorem \ref{kolmth}, since Brownian motion indexed by lower layers of $[0,1]^2$ does not have a continuous modification, as can be seen for instance in \cite{Ivanoff}.

\begin{example}\label{ex:layers}
Let $\mathcal{A}$ be the collection of lower layers of $[0,1]^2$, i.e. the subsets $A$ of $[0,1]^2$ such that $\forall t\in A$, $[0,t] \subseteq A$. 
For all $n\in\N$, let $\mathcal{A}_n$ be the collection of finite unions of sets in the dissecting collection of the diadic rectangles of $[0,1]^2$, i.e.
\begin{equation*}
\mathcal{A}_n = \left\{ \bigcup_{finite} [0,x] : 2^n x \in \mathbf{Z}^2 \cap (0,2^n]^2 \right\} \cup \{0\} \cup \{\emptyset \}.
\end{equation*}
Then, it can be shown that the cardinal $k_n$ of $\mathcal{A}_n$ satisfies $k_n\geq 2^{2^n}$ for all $n\in\N$. For all $U\in \Acal_n$, we can see that $\inf_{V\in \Acal_n, V\supsetneq U} d_{\lambda}(U,V) = 2^{-2n}$, hence there does not exist any $\qA$ such that $2^{-2n}$ and $k_n^{-1/\qA}$ are of the same order. 
Consequently the subclasses $(\Acal_n)_{n\in\N}$ cannot verify Assumption (\HA).

\end{example}

This final example gives a non-integer $\qA$ and should illustrate why our framework is fundamentally more general than the multiparameter framework.

\begin{example}[Indexing collection with $\qA\notin \N$]\label{ex:non-integer exponent}
We give an example of an indexing class $\Acal$ on a measure space $(\mathcal{T},\mu)$ admitting a non-integer $\qA$ which turns out to be equal to the Hausdorff dimension of $\mathcal{T}$.

In the sequel, we shall denote by $\mu_\varphi$ the Hausdorff measure defined from some gauge function $\varphi$ (see the book \cite{Falconer} for definitions). For a measurable set $A$, we say that $A$ has exact Hausdorff mesure $\mu_\varphi$ if $0<\mu_\varphi(A)<\infty$.
It is known (\cite[p.177]{xiao2009}) that the graph of the L\'evy fractional Brownian motion from $\R^N$ to $\R$ with Hurst parameter $H$ has almost surely an exact Hausdorff measure which is given by the gauge function:
\begin{equation*}
\varphi(r) = r^{N+1-H} (\log\log 1/r)^{\frac{H}{N}}\ .
\end{equation*}
So let us fix an $\omega$ such that $\varphi$ is indeed an exact Hausdorff measure for $\mathcal{T}\equiv\text{Gr}(X) = \{(s,X_s(\omega)),\ s\in [0,1]^N\}$. We define $\Acal$ as the collection of subsets of the form $U_{t,x} = \{(s,X_s(\omega)),\ s\in [0,t]^N,\ |X_s(\omega)|\leq x\}$ for $t\in[0,1]^N$ and $x\in[0,1]$ (assuming without restriction that $\sup_{t\in[0,1]^N} X_t(\omega) = 1$). The approximating finite subcollection $\Acal_n$ is defined, for any $n\in \N^*$, by the sets $A_{k,j}^{(n)} = U_{2^{-n} k, 2^{-n}j},\ k\in \{0,\dots,2^n-1\}^N, j\in\{0,\dots, J_n\}$, where $J_n$ is some integer (see Appendix \ref{app:0}). \\
It follows that the distance $d_{\mu_\varphi}(A_{k,j}^{(n)},A_{k+1,j+1}^{(n)}) = \mu_\varphi(A_{k+1,j+1}^{(n)}\setminus A_{k,j}^{(n)})$ between a ``dyadic'' set and its nearest neighbour is given by:
\begin{align*}
\mu_\varphi\left(\{(s,X_s(\omega)),\ s\in [0,2^{-n}(k+1)]\setminus [0,2^{-n}k],\ 2^{-n}j< X_s(\omega)\leq 2^{-n}(j+1)\}\right) 
\end{align*}
which is bounded by a constant times $k_n^{-1/(N+1-H)}$ (see App. \ref{app:0}).
Hence $\qA = N+1-H$.
\end{example}

\subsection{Kolmogorov's criterion}\label{sec:Kolm}
 
For any deterministic function $f:\mathcal{A}\rightarrow\textbf{R}$,
let us consider the {\em modulus of continuity} on any totally bounded
$\mathcal{B}\subset\mathcal{A}$
\[
\omega_{f,\mathcal{B}}(\delta)=\sup_{U,V\in\mathcal{B} \atop
\dA(U,V)\leq\delta} |f(U)-f(V)|\ ,\quad \delta>0\ .
\]

\noindent Recall that the function $f$ is said {\em H\"older continuous of order 
$\alpha>0$} if for all totally bounded $\mathcal{B}\subset\mathcal{A}$
one of the following equivalent conditions holds (e.g. see \cite{davar}, Chapter~$5$)
\begin{enumerate}[(i)]
\item $\displaystyle\limsup_{\delta\rightarrow 0}\delta^{-\alpha}.\omega_{f,\mathcal{B}}(\delta)
<\infty.$
\item There exists $M>0$ and $\delta_0>0$ 
such that 
for all $U,V\in\mathcal{B}$ with $\dA(U,V)<\delta_0$, 
$|f(U)-f(V)|\leq M.\dA(U,V)^{\alpha}$.
\end{enumerate}

For any general set-indexed Gaussian process, Dudley's Corollary 2.3 in \cite{dudley} allows to compute a modulus of continuity (giving the same kind of result than following Corollary \ref{GaussKolm}). This result holds under certain entropic conditions on the indexing collection, which are not so different from (\ref{eq:assump1}) since in most concrete examples, verifying an entropic condition amounts to finding an upper bound for the distance induced by the process.  
Assumption (\HA) and more precisely (\ref{eq:assump2}) allow to prove a continuity criterion in the non-Gaussian case. 
The following Theorem \ref{kolmth} does so in the general set-indexed framework of Ivanoff and Merzbach.


\begin{definition}\label{def:contract}
A (pseudo-)distance $\dA$ on $\Acal$ is said:
\begin{enumerate}[(i)]
\item\label{cond:OC} {\em Outer-continuous} 
if for any non-increasing sequence $(U_n)_{n\in\N}$ in $\Acal$ converging to $U=\bigcap_{n\in\N}U_n \in\Acal$,
$\dA(U_n, U)$ tends to $0$ as $n$ goes to $\infty$ ;

\item\label{cond:contract} {\em Contractive} if it is outer-continuous and if for any $U, V, W \in \Acal$,
\begin{equation*}
\dA(U\cap W,V\cap W) \leq \dA(U,V) .
\end{equation*}
 
\end{enumerate}
\end{definition}

\begin{remark}
The most important metrics in the context of set-indexed processes, $d_m$ and $d_H$, are contractive.
\end{remark}

\begin{theorem}\label{kolmth}
Let $\dA$ be a contractive (pseudo-)distance on the indexing collection $\mathcal{A}$, whose subclasses $\underline{\Acal}=(\Acal_n)_{n\in\N}$ satisfy Assumption (\HA) with a discretization exponent $\qA>0$. Let $X=\left\{ X_U ;\; U\in\mathcal{A} \right\}$ be a set-indexed process
such that
\begin{equation}\label{eq:hypoProcess}
\forall U,V\in\mathcal{A},\quad
\EE\left[|X_U-X_V|^{\alpha}\right] \leq
K\ \dA(U,V)^{\qA+\beta}
\end{equation}
where $K$, $\alpha$ and $\beta$ are positive constants.\\
Then, the sample paths of $X$ are almost surely locally $\gamma$-H\"older continuous 
for all $\gamma\in (0,\frac{\beta}{\alpha})$, i.e.
there exist a random variable $h^*$ and a constant $L>0$ 
such that almost surely
\begin{equation*}
\forall U,V\in\mathcal{A},\qquad \dA(U,V)<h^* \Rightarrow
|X_U - X_V| \leq L \ \dA(U,V)^{\gamma}.
\end{equation*}
\end{theorem}

\begin{proof}
Let us fix $\gamma\in (0,\frac{\beta}{\alpha})$ and denote $\mathcal{D}=\bigcup_{n\in\N}\mathcal{A}_n$ a countable dense subset of $\Acal$. First, let $(a_j)_{j\in \N}$ be any sequence of positive real numbers such that $\sum_{j\in\N}a_j<+\infty$, and for $n\in \N$ such that $\sum_{j\geq n} a_j \leq 1$, we have
\begin{align}\label{eq:proofKolm1}
\PP\bigg(\sup_{U\in \mathcal{D}} |X_U-X_{g_n(U)}|&\geq k_{n+1}^{-\gamma/\qA} \bigg) \nonumber\\
&\leq \PP\left(\exists U\in \mathcal{D},\ \sum_{j=n}^{\infty} |X_{g_{j+1}(U)}-X_{g_j(U)}| \geq k_{n+1}^{-\gamma/\qA}\right) \nonumber\\
&\leq \PP\bigg(\exists U\in \mathcal{D}, \exists j\geq n, |X_{g_{j+1}(U)}-X_{g_j(U)}|\geq a_jk_{n+1}^{-\gamma/\qA}\bigg) \\
&\leq \PP\left(\exists j\geq n,\ \exists V\in \Acal_{j+1},\ |X_{V}-X_{g_j(V)}| \geq a_j k_{n+1}^{-\gamma/\qA}\right) \nonumber \\
&\leq \sum_{j=n}^{\infty} \sum_{V\in \Acal_{j+1}} \PP\left(|X_{V}-X_{g_j(V)}| \geq a_j k_{n+1}^{-\gamma/\qA}\right). \nonumber
\end{align}

\noindent Now applying successively Tchebyshev's inequality, (\ref{eq:hypoProcess}) and Equation (\ref{eq:assump1}) of Assumption (\HA),
\begin{align*}
\PP\left(\sup_{U\in \mathcal{D}} |X_U-X_{g_n(U)}| \geq k_{n+1}^{-\gamma/\qA} \right) 
&\leq  \sum_{j=n}^{\infty} k_{j+1}  a_j^{-\alpha} k_{n+1}^{\alpha \gamma/\qA} \ \sup_{V\in \Acal_{j+1}}\EE\left[|X_V-X_{g_j(V)}|^{\alpha}\right] \\
&\leq K\ k_{n+1}^{\alpha \gamma/\qA} \sum_{j=n}^{\infty} a_j^{-\alpha} k_{j+1} \sup_{V\in \Acal_{j+1}} \dA(V,g_j(V))^{\qA + \beta} \\
&\leq K\ k_{n+1}^{\alpha \gamma/\qA} \ \sum_{j=n}^{\infty} a_j^{-\alpha} \frac{k_{j+1}}{k_j} k_{j}^{-\beta/\qA}.
\end{align*}

\noindent The admissibility of $(k_n)_{n\in \N}$ implies that for $\delta>0$, and for $n$ large enough (depending on $\delta$), $k_{n+1}^{\alpha \gamma/\qA}\leq (k_n^{\alpha \gamma/\qA})^{1+\delta}$, so that:
\begin{align*}
\PP\left(\sup_{U\in \mathcal{D}} |X_U-X_{g_n(U)}| \geq k_{n+1}^{-\gamma/\qA} \right) &\leq K\ k_{n}^{\delta \alpha \gamma/\qA} \ \sum_{j=n}^{\infty} a_j^{-\alpha} \frac{k_{j+1}}{k_j} k_{j}^{-\beta/\qA} k_n^{\gamma \alpha/\qA} \\
&\leq K\ k_{n}^{\delta \alpha \gamma/\qA} \ \sum_{j=n}^{\infty} a_j^{-\alpha} \frac{k_{j+1}}{k_j} k_{j}^{-(\beta-\gamma \alpha)/\qA}.
\end{align*}

\noindent Since $\beta-\alpha.\gamma>0$, $(a_j^{\alpha})_{j\in \N}$ can be chosen equal to $(k_j^{-(\beta- \alpha \gamma)/3 \qA})_{j\in \N}$ (which is indeed summable because $k_n$ is admissible), and then:
\begin{align*}
\PP\left(\sup_{U\in \mathcal{D}} |X_U-X_{g_n(U)}| \geq k_{n+1}^{-\gamma/\qA} \right) \leq K\ k_{n}^{\delta \alpha \gamma/\qA} \ \sum_{j=n}^{\infty} \frac{k_{j+1}}{k_j} k_{j}^{-2(\beta-\gamma \alpha)/3\qA},
\end{align*}
which finally leads to, for $\delta= (\beta - \alpha \gamma)/(6\alpha \gamma)$,
\begin{align*}
\PP\left(\sup_{U\in \mathcal{D}} |X_U-X_{g_n(U)}| \geq k_{n+1}^{-\gamma/\qA} \right) \leq K\ k_{n}^{-\delta \alpha \gamma/\qA} \ \sum_{j=n}^{\infty} \frac{k_{j+1}}{k_j} k_{j}^{-(\beta-\gamma \alpha)/3\qA}.
\end{align*}

\noindent Thus, this probability is summable and Borel-Cantelli's theorem implies the existence of $\Omega^{*}\subset\Omega$ with $\PP(\Omega^*)=1$ such that $\forall\omega\in\Omega^{*}$,
\begin{align}\label{eq:BC1}
\exists n^{*}(\omega)\in\N,\; \forall n\geq n^{*}, \forall U\in\mathcal{D}, \quad
|X_U-X_{g_n(U)}| < k_{n+1}^{-\gamma/\qA}.
\end{align}

\noindent Now, we develop the same argument for the following probability:
\begin{align}\label{eq:pairs}
\PP\left(\sup_{U\in \Acal_n} \sup_{V\in \mathcal{V}_n(U)} |X_U - X_V| \geq k_{n+1}^{-\gamma/\qA}\right) &\leq k_n \ N_n\ \sup_{U\in \Acal_n} \sup_{V\in \mathcal{V}_n(U)} \PP\left(|X_U - X_V| \geq k_{n+1}^{-\gamma/\qA}\right) \nonumber\\
&\leq K \ N_n \ k_{n+1}^{\alpha\gamma/\qA}\ k_n^{-\beta/\qA} \nonumber\\
&\leq K\ N_n\ k_n^{-(\beta-\alpha\gamma)/2\qA},
\end{align}
where we used $\delta$ as in the previous paragraph. This is summable by (\ref{eq:assump2}), hence there exists $\Omega^{**}$ a measurable subset of $\Omega$ of probability $1$ and $n^{**}$ a integer-valued finite random variable  such that on $\Omega^{**}$:
\begin{align}\label{eq:BC2}
\forall n\geq n^{**}, \quad
\sup_{U\in \Acal_n} \sup_{V\in \mathcal{V}_n(U)} |X_U-X_V| < k_{n+1}^{-\gamma/\qA}.
\end{align}

\noindent For any $U,V \in \mathcal{D}$, there is a unique $n\in \N$ such that $k_{n+1}^{-1/\qA}\leq \dA(U,V) < k_n^{-1/\qA}$. Let $I_n = [k_{n+1}^{-1/\qA},k_n^{-1/\qA})$. Without any restriction, we assume that $U\subseteq V$. Indeed, if this not the case, we shall consider $X_{U}- X_{V} = X_{U}-X_{U\cap V}+X_{U\cap V} - X_{V}$, where $\dA(U,U\cap V) \leq \dA(U,V)$ by contractivity. Since this implies that $g_n(V)\in \mathcal{V}_n(g_n(U))$, we will write, on $\Omega^*\cap \Omega^{**}$, for any $n\geq n^*\vee n^{**}$:
\begin{align}\label{eq:finChain}
\sup_{\substack{U,V\in \mathcal{D}\\ \dA(U,V) \in I_n}} |X_U -X_V| &\leq \sup_{\substack{U,V\in \mathcal{D}\\ \dA(U,V)\in I_n}}\left( |X_U - X_{g_n(U)}|\ + |X_{g_n(U)} - X_{g_n(V)}| + |X_{g_n(V)} -X_V| \right) \nonumber\\
&\leq 3 \ k_{n+1}^{-\gamma/\qA} \\
&\leq 3 \ \dA(U,V)^{\gamma} , \nonumber
\end{align}
as a consequence of Equations (\ref{eq:BC1}) and (\ref{eq:BC2}). Since $\Omega^*\cap \Omega^{**}$ is of probability $1$, we have proved that there exist a constant $L>0$ and a random variable $h^*$ such that
\begin{align}\label{eq:resultD}
\forall U,V\in \mathcal{D};\quad \dA(U,V)<h^*\Rightarrow
|X_U-X_V|\leq L\ \dA(U,V)^{\gamma} \quad\textrm{a.s.}
\end{align}


In the last part of the proof, we need to extend (\ref{eq:resultD}) to the whole class $\mathcal{A}$.
From the outer-continuity of $\dA$, we can claim: \\
On $\Omega^*$, for all $\epsilon \in (0,h^*)$, for all $U$ and $V$ in $\mathcal{A}$ with $\dA(U, V)<h^*-\epsilon$,
there exists $n_0>n^{*}$ such that $\dA(g_n(U), g_m(V))<h^*$ for all $n\geq n_0$ and $m\geq n_0$.
Thus by (\ref{eq:resultD}),
\begin{equation}\label{resultg_n}
\forall n>n_0, \forall m>n_0;\quad
|X_{g_n(U)}-X_{g_m(V)}|\leq L\ \dA(g_n(U), g_m(V))^{\gamma}.
\end{equation}

\noindent We define the process $\tilde{X}$ by
\begin{itemize}
\item $\forall\omega\notin\Omega^{*}$, $\forall U\in\mathcal{A}$, 
$\tilde{X}_U(\omega)=0$,
\item $\forall\omega\in\Omega^{*}$, 
\begin{itemize}
\item $\forall U\in \mathcal{D}$, $\tilde{X}_U(\omega)=X_U(\omega)$
\item $\forall U\in\mathcal{A}\setminus \mathcal{D}$,
$\tilde{X}_U(\omega)=\lim_{n\rightarrow\infty}X_{g_n(U)}(\omega)$.
\end{itemize}
\end{itemize}
Applying (\ref{resultg_n}) with $V=U$, the outer-continuity property of $\dA$ implies that $\left(X_{g_n(U)}(\omega)\right)_{n\in\N}$ is a Cauchy sequence and then converges in $\R$. 

\noindent The process $\tilde{X}$ satisfies almost surely
\begin{equation*}
\forall U,V\in\mathcal{A};\quad \dA(U, V)<h^*\Rightarrow
|\tilde{X}_U-\tilde{X}_V| \leq L.\;\dA(U, V)^{\gamma}.
\end{equation*}

\noindent Moreover,
\begin{itemize}
\item $\forall U\in \mathcal{D}$, $\tilde{X}_U=X_U$ almost surely.
\item $\forall U\in\mathcal{A}\setminus \mathcal{D}$, by construction,
$X_{g_n(U)}\stackrel{\textrm{a.s.}}{\rightarrow}\tilde{X}_U$ as $n\rightarrow\infty$.\\
\end{itemize}
Since $\EE\left[|X_{g_n(U)}-X_U|^{\alpha}\right]$ converges to $0$ when $n\rightarrow\infty$, the sequence $\left(X_{g_n(U)}\right)_{n\in\N}$ converges in probability to $X_U$. So, there exists a subsequence converging a.s.\\
From these two facts, we get $\tilde{X}_U=X_U$ a.s.
\end{proof}

The following example illustrates how the previous theorem can be applied to a non-Gaussian process whose H\"older regularity could not be obtained either by means of the theory of Ledoux and Talagrand \cite{LedouxTalagrand}(since it does not allow to treat H\"older regularity of general stochastic processes), nor by the Kolmogorov's criterion of \cite{davar} where the H\"older regularity is treated in the Euclidean space.

\begin{example}[Harmonizable fractional stable motion]\label{ex:Hfsm}
We define a set-indexed version of the harmonizable fractional stable motion (Hfsm), which is a process with fat tails (see e.g. \cite{xiao2010}).
Let $\Acal$ be an indexing collection on a measurable space $(\mathcal{T},m)$, such that $\qA\in(1,2)$ (that can be constructed as in Example \ref{ex:non-integer exponent}). Let $H\in (0,1)$ and $\alpha\in [1,2)$. Let $X$ be an $\Acal$-indexed process such that for any $U\in \Acal$, $X_U$ is an $\alpha$-stable random variable, and such that for any $\gamma<\alpha$:
\begin{equation*}
\EE\left(|X_U - X_V|^\gamma\right) \leq d_m(U,V)^{\gamma H} \ ,\quad U,V\in \Acal\ .
\end{equation*}
The previous theorem implies that $X$ is almost $H-\qA/\alpha$ H\"older continuous on $\Acal$, as long as we choose $\alpha H>\qA$.

\noindent For the sake of completeness, we sketch the construction of this process in Appendix \ref{app:0}.
\end{example}


As in the multiparameter's case, a simpler statement holds for Gaussian
processes.

\begin{corollary}\label{GaussKolm}
Let $\dA$ be a (pseudo-)distance on the indexing collection $\mathcal{A}$, whose subclasses $\underline{\Acal}=(\Acal_n)_{n\in\N}$ satisfy Assumption (\HA). 
Let $X=\left\{ X_U ;\; U\in\mathcal{A} \right\}$ be a centered Gaussian set-indexed 
process such that
\begin{equation*}
\forall U,V\in\mathcal{A},\quad
\EE\left[|X_U-X_V|^2\right] \leq K \ \dA(U, V)^{2\beta}
\end{equation*}
where $K>0$ and $\beta>0$. \\
Then, the sample paths of $X$ are almost surely locally $\gamma$-H\"older continuous 
for all $\gamma\in (0,\beta)$.
\end{corollary}


\begin{remark}\label{rem:hypGauss}
The proof of Theorem \ref{kolmth} shows that when Condition (\ref{eq:assump2}) is removed from Assumption (\HA), the conclusion remains true when the hypothesis (\ref{eq:hypoProcess}) is strengthened in
\begin{equation*}\label{eq:hypoProcess0}
\forall U,V\in\mathcal{A};\quad
\EE\left[|X_U-X_V|^{\alpha}\right] \leq
K\ \dA(U,V)^{2\qA+\beta}.
\end{equation*}
The result follows from the simple estimation $N_n\leq k_n$ in Equation (\ref{eq:pairs}).
In that case, the validity of Corollary \ref{GaussKolm} persists, since the integer $p$ can be chosen such that $2p\beta > 2\qA$ (instead of $2p\beta > \qA$).
\end{remark}

\

As previously mentioned, the Brownian motion indexed by the lower layers of $[0,1]^2$ is discontinuous with probability one (e.g. see Theorem 1.4.5 in \cite{adlerTaylor}). 
The previous Theorem~\ref{kolmth} and Corollary \ref{GaussKolm} do not contradict this fact, since the collection of lower layers of $[0,1]^2$ do not satisfy Assumption (\HA) (see Example \ref{ex:layers}).
This latter result is improved by the following corollary of Theorem \ref{kolmth}.

\begin{corollary}
Any subclasses $(\Acal_n)_{n\in\N}$ satisfying Condition (4) of Definition~\ref{basic} for the indexing collection of lower layers of $[0,1]^2$ do not satisfy Assumption (\HA).
\end{corollary}


\begin{remark}
In statistics, one often considers empirical processes on Vapnik-\v{C}ervonenkis sets (see \cite{adlerTaylor}). On such a set, the metric entropy has actually a very similar bound to the one obtained with our assumption, so that Gaussian processes can be treated similarly. This is important in particular for the Brownian bridges arising as limiting processes in nonparametric Central limit theorems.
\end{remark}

\subsection{$\mathcal{C}$-increments}\label{sec:Cinc-Kolm}

So far, we only considered simple {\em increments} of $X$ of the form $X_U - X_V$ for $U,V\in \Acal$ not necessarily ordered. 
However these quantities do not constitute the only natural extension of the one-parameter $X_t-X_s$ ($s,t\in\R_+$) to multiparameter (e.g. \cite{davar,AMSW,herbin}) and set-indexed (e.g. \cite{Ivanoff,ehem2}) settings, particularly when increment stationarity property is concerned. 
This section is devoted to usual increments of set-indexed processes, which extend the rectangular increments of multiparameter processes. Let us define, for any given indexing collection $\Acal$, the collection $\Ccal$ of subsets of $\mathcal{T}$, defined as 
\begin{equation*}
\mathcal{C} = \{U_0\setminus \cup_{i=1}^k U_i; \ U_0, U_1, \dots, U_k\in \Acal, k\in \N\}.
\end{equation*}
This collection is used to index the process $\Delta X$, defined by $\Delta X_C = X_{U_0} - \Delta X_{U_0 \cap \bigcup_{i\geq 1} U_i}$ for $C=U_0\setminus \cup_{i=1}^k U_i$, where $\Delta X_{U_0 \cap \bigcup_{i\geq 1} U_i}$ is given by the {\em inclusion-exclusion}  formula
\begin{equation}\label{eq:inc}
\Delta X_{U_0 \cap \bigcup_{i\geq 1} U_i} = \sum_{i=1}^k \sum_{j_1<\dots<j_i} (-1)^{i-1} X_{U_0 \cap U_{j_1} \cap \dots \cap U_{j_i}}.
\end{equation}
The existence of the increment process $\Delta X$ indexed by $\Ccal$ requires that for any $C\in\mathcal{C}$, the value $\Delta X_C$ does not depend on the representation of $C$.

\begin{corollary}\label{corCl}
Under the hypotheses of Theorem \ref{kolmth} and if the distance $\dA$ on the class $\Acal$ is assumed to be contractive, for each fixed integer $l\geq 1$, for all $\gamma\in (0,\beta/\alpha)$, there exist a random variable $h^{**}$ and a constant $L>0$ such that, with probability one,
\begin{align}
\label{eqHolderCl}
\forall C=U\setminus \bigcup_{i\leq l} V_i \textrm{ with } &U, V_1, \dots, V_l\in\Acal, \nonumber\\
&\max_{i\leq l} \{m(U\setminus V_i)\}<h^{**} \Rightarrow
|\Delta X_C| \leq L \ m(C)^{\gamma}.
\end{align}
\end{corollary}

\noindent For a proof of this result, see Appendix \ref{app:1}.

\medskip

Corollary \ref{corCl}, as a result on the class $\mathcal{C}^l = \left\{ U\setminus V;\; U \in \Acal, V\in \mathcal{B}^l \right\}$ where \linebreak $\mathcal{B}^l = \left\{\bigcup_{i=1}^l V_i;\; V_1,\dots, V_l \in \mathcal{A} \right\}$, does \emph{not} extend to the whole $\Ccal=\bigcup_{l\geq 1}\Ccal^l$, as the following example shows. 
The next result is an adaptation of an example in \cite{adlerTaylor,Ivanoff} to the set-indexed setting. It states that the Brownian motion can be unbounded on $\Ccal$ when $\Acal$ is the collection of rectangles of $[0,1]^2$. 

\begin{proposition}\label{prop:exBM}
Let $W$ be a Brownian motion indexed by the Borelian sets of $[0,1]^2$, i.e. a centered Gaussian process with covariance structure
\begin{equation*}
\EE[W_C W_{C'}] = \lambda (C \cap C'),\quad\forall C,C' \in \mathcal{B}([0,1]^2)
\end{equation*}
where $\lambda$ denotes the Lebesgue measure.\\
Let $\Acal$ be the collection of rectangles of $[0,1]^2$. In the sequel, we consider the restriction on the class $\Ccal$, related to $\Acal$, of the Brownian motion defined above.\\
Then for all $h>0$, all $M>0$, and for almost all $\omega\in\Omega$, there exist sequences of sets $\left(C_n(\omega)\right)_{n\in\N}$, $\left(C'_n(\omega)\right)_{n\in\N}$ in $\Ccal$ such that $\lambda(C_n(\omega)) \vee \lambda(C'_n(\omega)) < h$ and for $n$ big enough, 
\begin{equation*}
\max \{|W_{C_n(\omega)}(\omega)|, |W_{C'_n(\omega)}(\omega)|\} > \frac{M}{8}.
\end{equation*}
\end{proposition}

Without any stronger condition than Assumption (\HA) on the sub-semilattices $(\mathcal{A}_n)_{n\in\N}$, the previous example of set-indexed Brownian motion dismisses a possible definition of the H\"older continuity for stochastic processes of the form:
\begin{equation*}
\exists M>0, \exists\delta_0>0:\ \forall C\in \mathcal{C}\ \textrm{with}\ m(C)<\delta_0,\ |\Delta X_C|\leq M.m(C)^{\alpha}.
\end{equation*}

\

\section{H\"older exponents for set-indexed processes}\label{sec:HoExp}

We localize the two expressions (i) and (ii) for H\"older-continuity (Section \ref{sec:Kolm}) which leads to two different notions. 
Indeed, for the distance $\dA$ on $\mathcal{A}$, if $B_{\dA}(U_0,\rho)$ (or simply $B(U_0,\rho)$ if the context is clear) denotes the open ball centered in $U_0\in\mathcal{A}$ and whose radius is $\rho>0$,
we get

\begin{enumerate}[(i)$_{loc}$]
\item
\[
\limsup_{\delta\rightarrow 0+} \delta^{-q}
\sup_{U,V\in B_{\dA}(U_0,\delta)}|f(U)-f(V)| < \infty.
\]
\item There exist $M>0$ and $\delta_0>0$ such that
\[
\forall U,V\in B_{\dA}(U_0,\delta_0),\quad
|f(U)-f(V)| \leq M\ \dA(U,V)^q.
\]
\end{enumerate}

\noindent Localizing around
$U_0\in\mathcal{A}$ only gives (ii)$_{loc}\Rightarrow$(i)$_{loc}$. Thus we consider two H\"older exponents at 
$U_0\in\mathcal{A}$:

\begin{itemize}
\item the pointwise H\"older exponent
\begin{equation}\label{eq:pointholder}
\alphar_f(U_0)=\sup\left\{ \alpha:\;\limsup_{\rho\rightarrow 0}
\sup_{U,V\in B(U_0,\rho)}
\frac{|f(U)-f(V)|}{\rho^{\alpha}}<\infty \right\},
\end{equation}
\item and the local H\"older exponent
\begin{equation}\label{eq:localholder}
\widetilde{\alphar}_f(U_0)=\sup\left\{ \alpha:\;\limsup_{\rho\rightarrow 0}
\sup_{U,V\in B(U_0,\rho)}
\frac{|f(U)-f(V)|}{\dA(U,V)^{\alpha}}<\infty \right\}.
\end{equation}
\end{itemize}
Each one allows to measure the regularity of the function $f$.
In general, we have 
\begin{equation}\label{ineqholder}
\tilde{\alphar}_f \leq \alphar_f \ ,
\end{equation}
but the inequality can be strict (see for instance the chirp function in \cite{ehjlv}). Hence, the sole pointwise exponent is not sufficient to describe the irregularity of the function.

In the case of Gaussian processes (see \cite{ehjlv}), we define
the {\em deterministic pointwise H\"older exponent}
\begin{equation}
\mathbb{\bbalpha}_X(U_0)=
\sup\left\{\sigma;\;\limsup_{\rho\rightarrow 0}\sup_{U,V\in B(U_0,\rho)}
\frac{\EE\left[X_U-X_V\right]^2}{\rho^{2\sigma}}<\infty
\right\}
\end{equation}
and the {\em deterministic local H\"older exponent}
\begin{equation}
\widetilde{\mathbb{\bbalpha}}_X(U_0)=
\sup\left\{\sigma;\;\limsup_{\rho\rightarrow 0}\sup_{U,V\in B(U_0,\rho)}
\frac{\EE\left[X_U-X_V\right]^2}{\dA(U,V)^{2\sigma}}<\infty
\right\}.
\end{equation}

\noindent On the space $(\R^N,\|.\|)$, it is shown in \cite{ehjlv} that for all $t_0\in\R^N_{+}$, the pointwise and local H\"older exponents of $X$ at $t_0$ satisfy almost surely
\begin{equation*}
\alphar_X(t_0)=\mathbb{\bbalpha}_X(t_0)\quad\textrm{and}\quad
\widetilde{\alphar}_X(t_0)=\widetilde{\mathbb{\bbalpha}}_X(t_0).
\end{equation*}

In the following sections, several other definitions are studied for H\"older regularity of set-indexed processes. They are connected to the various ways to study the local behaviour of the sample paths of $X$ around a given set $U_0\in\mathcal{A}$.

\subsection{Separability of stochastic processes}
Defining H\"older exponents by expressions (\ref{eq:pointholder}) and (\ref{eq:localholder}) leads us to ask whether they are random variables, in order to consider measurable events related to these quantities. This question was first answered by Doob (see \cite{Doob54}) for linear parameter space, see \cite{davar} for a contemporary exposition.

\begin{definition}[\cite{Doob54}]
A process $\{X_U, U\in\Acal\}$ is said \emph{separable}  if there exist an at most countable collection $\mathcal{S} \subset \Acal$ and a null set $\Lambda$ such that for all closed sets $F\subset \R$ and all open set $\mathcal{O}$ for the topology induced by $\dA$,
\begin{equation*}
\left\{ \omega: X_U(\omega) \in F \textrm{ for all } U\in \mathcal{O}\cap \mathcal{S} \right\} \setminus \left\{ \omega: X_U(\omega) \in F \textrm{ for all } U\in \mathcal{O} \right\} \subset \Lambda
\end{equation*}
\end{definition}

This definition is well suited for set-indexed processes since we have the following:

\begin{theorem}[from {\cite[Theorem 2 p.153]{Gikhman}}]
Any stochastic process from a separable metric space with values in a locally compact space admits a separable modification. In particular, if the sub-collections $(\Acal_n)_{n\in\N}$ and the metric $\dA$ satisfy Assumption~(\HA), any $\R$-valued set-indexed stochastic process $X=\{X_U; \ U\in \Acal\}$ has a separable modification.
\end{theorem}

We shall now consider that all our processes are separable, so that all the random H\"older exponents defined here are indeed random variables.

\subsection{Definition of H\"older exponents on $\Ccal^l$}\label{sec:HoCl}

Following expression (\ref{eq:inc}) for the definition of the increments of a set-indexed process, we consider alternative definitions for H\"older exponents, where the quantities $X_U-X_V$ are substituted with $\Delta X_{U\setminus V}$.

As stated in Section \ref{sec:Cinc-Kolm} and illustrated in Proposition \ref{prop:exBM}, it is not wise to consider $\Delta X_{U\setminus V}$ when $U\in\mathcal{A}$ and $V\in\mathcal{A}(u)$ are close to a given $U_0\in\mathcal{A}$.

\noindent Thus, we fix any integer $l\geq 1$ and set for all $U\in\Acal$ and $\rho>0$,
\begin{equation*}
\mathcal{B}^l(U, \rho) = \left\{\bigcup_{1\leq i\leq l} V_i ;\; V_1,\dots,V_l\in\Acal, 
\ \max_{1\leq i\leq l}\dA(U,V_i)< \rho\right\}.
\end{equation*} 
The pointwise and local H\"older $\mathcal{C}^l$-exponents at $U_0\in\Acal$ are respectively defined as
\begin{equation*}
\alphar_{X,\mathcal{C}^l}(U_0) = \sup \left\{ \alpha: \limsup_{\rho \rightarrow 0} \sup_{\substack{U\in B_{\dA}(U_0,\rho) \\ V \in \mathcal{B}^l(U_0,\rho)}} \frac{|\Delta X_{U\setminus V}|}{\rho^{\alpha}} < \infty \right\},
\end{equation*}
and
\begin{equation*}
\widetilde{\alphar}_{X,\mathcal{C}^l}(U_0) = \sup \left\{ \alpha: \limsup_{\rho \rightarrow 0} \sup_{\substack{U\in B_{\dA}(U_0,\rho) \\ V \in \mathcal{B}^l(U_0,\rho)}} \frac{|\Delta X_{U\setminus V}|}{\dA(U,V)^{\alpha}} < \infty \right\}.
\end{equation*}

The following result shows that the $\mathcal{C}^l$-exponents do not depend on $l$ and, consequently, they provide a definition of H\"older exponents on the class $\mathcal{C}$. Moreover, these exponents can be compared to the exponents defined by (\ref{eq:pointholder}) and (\ref{eq:localholder}).

\begin{proposition}\label{propComparaisonAlphas}
If $\dA$ is a contractive distance, for any $U_0\in\mathcal{A}$, the exponents $\alphar_{X,\mathcal{C}^l}(U_0)$ and $\widetilde{\alphar}_{X,\mathcal{C}^l}(U_0)$ do not depend on the integer $l\geq 1$.
They are denoted by $\alphar_{X,\mathcal{C}}(U_0)$ and $\widetilde{\alphar}_{X,\mathcal{C}}(U_0)$ respectively.

\noindent Moreover, for all $U_0\in\mathcal{A}$ and all $\omega\in\Omega$,
\begin{align*}
\alphar_{X,\mathcal{C}}(U_0)(\omega) \geq \alphar_X(U_0)(\omega) 
\quad\textrm{and}\quad
\widetilde{\alphar}_{X,\mathcal{C}}(U_0)(\omega) \geq \widetilde{\alphar}_X(U_0)(\omega).
\end{align*}
\end{proposition}

\begin{proof}
We only detail the case of the pointwise exponent. The proof for the local exponent is totally similar. \\
From the definition of the $\mathcal{C}^l$-exponents, since $l\geq l'$ implies $\mathcal{B}^{l'}(U_0, \rho)\subseteq\mathcal{B}^l(U_0, \rho)$, it is clear that
\begin{equation*}
\forall \omega\in\Omega,\forall l\geq l',\quad
\alphar_{X,\mathcal{C}^l}(U_0)(\omega) \leq \alphar_{X,\mathcal{C}^{l'}}(U_0)(\omega).
\end{equation*}
For the sake of readability, we prove the converse inequality for $l=2, l'=1$ (the other cases are very similar). For any $\rho>0$, let $U\in B_{\dA}(U_0, \rho)$, and $V=V_1\cup V_2\in\mathcal{B}^l(U_0,\rho)$ with $V_1,V_2\in\mathcal{A}$. 
From the inclusion-exclusion formula,
\begin{align*}
|\Delta X_{U\setminus V}| &= |X_U - X_{U\cap V_1} - X_{U \cap V_2} + X_{U\cap V_1 \cap V_2}| \\
&= |\Delta X_{U\setminus V_1} + \Delta X_{U\setminus V_2} - \Delta X_{U\setminus (V_1 \cap V_2)}|\\
&\leq |\Delta X_{U\setminus V_1}|+ |\Delta X_{U\setminus V_2}| + |\Delta X_{U\setminus (V_1 \cap V_2)}|.
\end{align*}
We have $\dA(U_0,V_1) \leq \rho$, $\dA(U_0,V_2) \leq \rho$ and 
\begin{align*}
\dA(U_0,V_1 \cap V_2)&\leq\dA(U_0,V_1)+\dA(V_1,V_1\cap V_2) \\
&\leq \dA(U_0,V_1)+\dA(V_1,V_2) \leq 2\dA(U_0,V_1)+\dA(U_0,V_2)\leq 3\rho,
\end{align*}
using $\dA(V_1,V_1\cap V_2)\leq\dA(V_1,V_2)$ from the contracting property of $\dA$.\\
Then, for all $\alpha < \alphar_{X,\mathcal{C}^{l'}}(U_0)(\omega)$, 
$$ \limsup_{\rho \rightarrow 0} \sup_{\substack{U\in B_{\dA}(U_0,\rho) \\ V \in \mathcal{B}^l(U_0,\rho)}} \frac{|\Delta X_{U\setminus V}|}{\rho^{\alpha}} < \infty, $$
which says that $\alpha < \alphar_{X,\mathcal{C}^{l}}(U_0)(\omega)$.
Thus, $\alphar_{X,\mathcal{C}^{l'}}(U_0)(\omega) \leq \alphar_{X,\mathcal{C}^{l}}(U_0)(\omega)$.\\
This inequality shows that $\alphar_{X,\mathcal{C}^{l}}(U_0)(\omega)$ does not depend on the integer $l\geq 1$.

For the second part of the Proposition, it suffices to prove the inequality for $l=1$. This is straightforward, since for a fixed $U\in B_{\dA}(U_0, \rho)$, 
\begin{align*}
\sup_{V \in \mathcal{B}^1(U_0,\rho)} |\Delta X_{U\setminus V}| 
\leq \sup_{W\in B_{\dA}(U_0,\rho)} |X_U - X_W|.
\end{align*}
Hence $\alphar_X(U_0)\leq \alphar_{X,\mathcal{C}^l}(U_0)$. The  inequality for the local exponent can be obtained identically, or one can notice that it is a direct consequence of Corollary \ref{corCl}.

\end{proof}

\begin{remark}
The previous definition of the pointwise H\"older exponent on $\Ccal^l$ is not equivalent to the quantity
\begin{equation*}
\sup \left\{ \alpha: \limsup_{\rho \rightarrow 0} \sup_{\substack{U\in B_{\dA}(U_0,\rho) \\ V \in \mathcal{B}^l:\ \dA(U_0, V) < \rho}} \frac{|\Delta X_{U\setminus V}|}{\rho^{\alpha}} < \infty \right\},
\end{equation*}
since one can verify easily that $\{V \in \mathcal{B}^l:\ d_{\lambda}(U_0, V) < \rho\}\neq \{V\in\mathcal{B}^l(U_0,\rho)\}$.
The problem with the latter is that it is not possible to control the quantity $|X_U - \Delta X_{V_1 \cup V_2}|$ using $|X_U - X_{V_1}|$, $|X_U - X_{V_1}|$ and $|X_U - X_{V_1\cap V_2}|$ as was done in the previous proofs.
\end{remark}

The arguments of the proof of Proposition \ref{propComparaisonAlphas} in the particular case of $l=1$ leads to: for all $\omega$,
\begin{equation*}
\alphar_{X,\Ccal}(U_0)(\omega) \geq \sup\left\{ \alpha:\;\limsup_{\rho\rightarrow 0}
\sup_{\substack{U,V\in B_{\dA}(U_0,\rho) \\ U\subset V}}
\frac{|X_U(\omega)-X_V(\omega)|}{\rho^{\alpha}}<\infty \right\},
\end{equation*}
and
\begin{equation*}
\widetilde{\alphar}_{X,\Ccal}(U_0)(\omega) \geq \sup\left\{ \alpha:\;\limsup_{\rho\rightarrow 0}
\sup_{\substack{U,V\in B_{\dA}(U_0,\rho) \\ U\subset V}}
\frac{|X_U(\omega)-X_V(\omega)|}{\dA(U,V)^{\alpha}}<\infty \right\}.
\end{equation*}
The converse inequalities follow from the fact that the set of $U,V\in B_{\dA}(U_0,\rho)$ with $U\subset V$ is included in the set of $U\in B_{\dA}(U_0,\rho)$ and $V\in \mathcal{B}^1(U_0,\rho)$.
Then, we can state:

\begin{corollary}\label{cor:expC}
If $\dA$ is a contractive distance, 
the pointwise and local H\"older $\mathcal{C}$-exponents at $U_0\in\Acal$ are respectively given by
\begin{equation*}
\alphar_{X,\Ccal}(U_0) = \sup\left\{ \alpha:\;\limsup_{\rho\rightarrow 0}
\sup_{\substack{U,V\in B_{\dA}(U_0,\rho) \\ U\subset V}}
\frac{|X_U-X_V|}{\rho^{\alpha}}<\infty \right\},
\end{equation*}
and
\begin{equation*}
\widetilde{\alphar}_{X,\Ccal}(U_0) = \sup\left\{ \alpha:\;\limsup_{\rho\rightarrow 0}
\sup_{\substack{U,V\in B_{\dA}(U_0,\rho) \\ U\subset V}}
\frac{|X_U-X_V|}{\dA(U,V)^{\alpha}}<\infty \right\}.
\end{equation*}
\end{corollary}

\

\subsection{Pointwise continuity}\label{sec:Hopc}

In \cite{ehem3}, a weak form of continuity is considered in the study of set-indexed Poisson process, set-indexed Brownian motion and more generally set-indexed L\'evy processes. 
In particular, the sample paths of the set-indexed Brownian motion are proved to be pointwise continuous as a set-indexed L\'evy process with Gaussian increments.
Notice that such a property does not require Assumption~(\HA) on $\Acal$.
We recall the following definitions:

\begin{definition}[\cite{ehem3}]
The \emph{point mass jump} of a set-indexed function $x: \Acal \rightarrow \R$ at $t\in \mathcal{T}$ is defined by
\begin{equation}\label{eq:Jtx}
J_t(x) = \lim_{n\rightarrow \infty} \Delta x_{C_n(t)}, \ \ \textrm{where } \ C_n(t) = \bigcap_{\substack{C\in \Ccal_n \\ t\in C}} C
\end{equation}
and for each $n\geq 1$, $\mathcal{C}_n$ denotes the collection of subsets $U\setminus V$ with $U\in\mathcal{A}_n$ and $V\in\mathcal{A}_n(u)$.
\end{definition}

\begin{definition}[\cite{ehem3}]
A set-indexed function $x: \Acal \rightarrow \R$ is said \emph{pointwise continuous} at $t\in \mathcal{T}$ if $J_t(x) = 0$.
\end{definition}


Let us recall that a subset $\mathcal{A}'$ of $\mathcal{A}$ which is closed under
arbitrary intersections is called a {\em lower sub-semilattice} of $\mathcal{A}$. 
The ordering of a lower sub-semilattice $\mathcal{A}'=\{A_1,A_2, \dots\}$ is said to be {\em consistent} if $A_i\subset A_j \Rightarrow i\leq j$. 
Proceeding inductively, we can show that any lower sub-semilattice admits a consistent ordering, which is not unique in general (see \cite{Ivanoff}). 
If $\{A_1,\dots,A_n\}$ is a consistent ordering of a finite lower sub-semilattice $\mathcal{A}'$, the set $C_i = A_i\setminus \bigcup_{j\leq i-1} A_j$ is called {\em the left neighbourhood} of $A_i$ in $\mathcal{A}'$. Since $C_i = A_i \setminus \bigcup_{A\in\mathcal{A}', A\nsubseteq A_i} A$, the definition of the left neighbourhood does not depend on the ordering.

\begin{proposition}\label{prop:crit-pc-unif}
Let $X=\left\{ X_U ;\; U\in\mathcal{A} \right\}$ be a set-indexed process and let $U_{\max}$ be a subset in $\Acal$ such that $m(U_{\max})<+\infty$ and assume that there exist $p>0$, $q>1$, $N\geq 1$ and $K>0$ such that for all $t\in U_{\max}$ and all $n\geq N$,
\begin{equation}
\label{eq:hypoProcess3}
\EE\big[|\Delta X_{C_n(t)}|^p\big] \leq K\ m(C_n(t))^{q} \ .
\end{equation}
Then, for any $\gamma\in (0,(q-1)/p)$, there exists an increasing function $\varphi:\N\rightarrow\N$ and a random variable $n^* \geq 1$ satisfying, with probability one,
\begin{equation*}
\forall t\in U_{\max}, \ \forall n\geq n^*, \quad 
|\Delta X_{C_{\varphi(n)}(t)}| \leq m(C_{\varphi(n)}(t))^{\gamma}.
\end{equation*}
\end{proposition}

\begin{proof}
Up to restricting the indexing collection to $\{U\cap U_{max}, U\in \Acal\}$, we assume in this proof that the indexing collection $\Acal$ is included in $U_{\max}$. \\
For all $0<\gamma<\displaystyle\frac{q-1}{p}$, we consider $\displaystyle S_n = \sup\left\{\frac{|\Delta X_{C_n(t)}|}{m(C_n(t))^{\gamma}};\; t\in U_{\max} \right\}$, where $C_n(t)$ is defined in (\ref{eq:Jtx}). 
When $t$ ranges $U_{\max}$, the subset $C_n(t)$ ranges $\Ccal^l(\Acal_n)$, the collection of the disjoint left-neighbourhoods of $\Acal_n$. 
Consequently we can write $\displaystyle S_n = \sup\left\{\frac{|\Delta X_{C}|}{m(C)^{\gamma}};\; C\in\Ccal^l(\Acal_n) \right\}$.

\noindent For any integer $p\geq 1$, we have
\begin{align*}
\PP(S_n \geq 1) &\leq \sum_{C\in \Ccal^l(\Acal_n)} \PP(|\Delta X_C| \geq m(C)^{\gamma})\\
&\leq \sum_{C\in \Ccal^l(\Acal_n)} \frac{\EE\left[|\Delta X_C|^{p}\right]}{m(C)^{\gamma p}}
\leq K\ \sum_{C\in \Ccal^l(\Acal_n)} m(C)^{q-\gamma p} \ .
\end{align*}

\noindent Since $q-\gamma p>1$, we have
\begin{align*}
\PP(S_n \geq 1) &\leq K\ \left(\sum_{C\in \Ccal^l(\Acal_n)} m(C) \right) \sup_{C\in \Ccal^l(\Acal_n)} \left\{m(C)^{q-\gamma p-1}\right\}\\
&\leq K \ m(U_{max}) \sup_{C\in \Ccal^l(\Acal_n)} \left\{m(C)^{q-\gamma p-1}\right\}
\end{align*}
where the fact that the $C\in \Ccal^l(\Acal_n)$ are disjoint is used. Up to choosing an extraction $\varphi$ for the sequence $u_n = \sup_{C\in \Ccal^l(\Acal_n)} \left\{m(C)^{q-\gamma p-1}\right\}$, we can assume that $u_n$ is summable.
Hence the Borel-Cantelli Lemma implies that for $0<\gamma < (q-1)/p$, $\{S_{\varphi(n)} < 1\}$ happens infinitely often, which gives the result.
\end{proof}

Note again that Proposition \ref{prop:crit-pc-unif} does not require Assumption (\HA). It leads naturally to another definition of the  local H\"older regularity of a set-indexed process:

\begin{definition}
The \emph{pointwise continuity H\"older exponent} at any $t\in \mathcal{T}$ is defined by
\begin{equation*}
\alphar_X^{pc}(t) = \sup \left\{ \alpha: \ \limsup_{n\rightarrow \infty} \frac{|\Delta X_{C_n(t)}|}{m(C_n(t))^{\alpha}} < \infty \right\}.
\end{equation*}
\end{definition}
\noindent According to Proposition \ref{prop:crit-pc-unif}, if $X$ is a $\Acal$-indexed process satisfying hypothesis (\ref{eq:hypoProcess3}), then with probability one, $\alphar_X^{pc}(t) \geq (q-1)/p$ for all $t\in U_{max}$.

\begin{remark}
As in the continuity criterion (Theorem \ref{kolmth} and Corollary \ref{GaussKolm}), the proof of Proposition \ref{prop:crit-pc-unif} can be improved for $\gamma \in (0,(kq-1)/kp)$ for any $k\in \N$, when the process is Gaussian. In that specific case, the upper bound for admissible values of $\gamma$ is $q/p$ (instead of $(q-1)/p$).
\end{remark}

\

\section{Connection with H\"older exponents of projections on flows}\label{sec:flow}

In this section, we consider the concept of \emph{flow}, which is a useful tool to reduce characterization or convergence problems to a one-dimensional issue. \emph{Flows} have been used to characterize: strong martingales \cite{Ivanoff}, set-Markov processes \cite{BaIv02}, set-indexed fractional Brownian motion \cite{ehem2} and set-indexed L\'evy processes \cite{ehem3}.

\begin{definition}[\cite{Ivanoff}]\label{flowdef}
An {\em elementary flow} is defined to be a continuous increasing function $f:[a,b]\subset\mathbf{R}_+\rightarrow\mathcal{A}$, i.e. such that
\begin{align*}
\forall s,t\in [a,b];\quad & s<t \Rightarrow f(s)\subseteq f(t)\\
\forall s\in [a,b);\quad & f(s)=\bigcap_{v>s}f(v)\\
\forall s\in (a,b);\quad & f(s)=\overline{\bigcup_{u<s}f(u)}.
\end{align*}

A {\em simple flow} is a continuous function $f:[a,b]\rightarrow\mathcal{A}(u)$ such that there exists a finite sequence $(t_0,t_1,\dots,t_n)$ with $a=t_0<t_1<\dots<t_n=b$ and elementary flows $f_i:[t_{i-1},t_i]\rightarrow\mathcal{A}$ ($i=1,\dots,n$) such that
\begin{equation*}
\forall s\in [t_{i-1},t_i];\quad
f(s)=f_i(s)\cup \bigcup_{j=1}^{i-1}f_j(t_j).
\end{equation*}
The set of all simple (resp. elementary) flows is denoted $S(\mathcal{A})$ (resp. $S^e(\mathcal{A})$).
\end{definition}

According to \cite{ehem2}, we use the parametrization of flows which allows to preserve the increment stationarity property under projection on flows (it avoids the appearance of a time-change).

\begin{definition}[\cite{ehem2}]
For any set-indexed process $X=\left\{X_U;\;U\in\mathcal{A}\right\}$ on the space $(\mathcal{T},\mathcal{A},m)$ and any simple flow $f:[a,b]\rightarrow\mathcal{A}(u)$, the {\em $m$-standard projection of $X$ on $f$} is defined as the process $$X^{f,m}=\left\{X^{f,m}_t=\Delta X_{f\circ\theta^{-1}(t)};\;t\in \theta([a,b])\right\},$$
where $\theta$ is the function $t\mapsto m[f(t)]$ and $\theta^{-1}$ its right inverse.
\end{definition}

The importance of flows in the study of set-indexed processes follows the fact that the finite dimensional distributions of an additive $\mathcal{A}$-indexed process $X$ determine and are determined by the finite dimensional distributions of the class $\{X^{f,m},\ f\in S(\mathcal A)\}$ (\cite{Ivanoff}).

As the projection of a set-indexed process on any flow is a real-parameter
process, its classical H\"older exponents can be considered and compared to the
exponents of the set-indexed process.
In the sequel, we study how regularity of flows connects the exponents $\alphar_X(U_0)$ (resp. $\widetilde{\alphar}_X(U_0)$) and $\alphar_{X^{f,m}}(t_0)$ (resp. $\widetilde{\alphar}_{X^{f,m}}(t_0)$), when $U_0\in\Acal$ and $f\circ\theta^{-1}(t_0)=U_0$.

\

For any $U_0 \in \mathcal{A}$, let us denote by $S(\Acal,U_0)$ the subset of $S(\Acal)$ containing all the simple flows $f:\theta^{-1}(I_f) \rightarrow \Acal(u)$ such that there exists $t_0 > 0$ satisfying $f\circ \theta^{-1}(t_0) = U_0$, and where $I_f$ is a closed interval of $\R_+$ containing a ball centered in $t_0$. Such a $t_0$ does not depend on the flow $f$, since $t_0 = m(U_0)$. 
In the same way, we define $S^e(\Acal, U_0)$ for elementary flows.

\begin{lemma}\label{lem:flow}
Let $f\in S(\Acal,U_0)$ and $\eta >0$ such that $B(t_0,\eta) \subset I_f$. 
For all $t\in B(t_0,\eta)$, $f\circ \theta^{-1}(t) \in B^{(u)}_{d_m}(U_0,\eta) = \{ A \in \Acal(u): m(A \bigtriangleup U_0) < \eta \}$. 
\end{lemma}

\begin{proof}
$\theta^{-1}(t) = \inf \{x\in I_f: \theta(x) \geq t\}$. As $\theta$ is increasing, $\theta^{-1}$ is increasing as well. We assume without loss of generality that $t\geq t_0$. Then,
\begin{eqnarray*}
d_m(f\circ \theta^{-1}(t),U_0) &=& m\left(f\circ \theta^{-1}(t)\bigtriangleup f\circ \theta^{-1}(t_0)\right) \\
&=& m\left(f\circ \theta^{-1}(t) \setminus f\circ \theta^{-1}(t_0)\right) \\
&=& m(f\circ \theta^{-1}(t)) - m(f\circ \theta^{-1}(t_0))\\
&=& t-t_0.
\end{eqnarray*}
\end{proof}

\noindent Using Lemma \ref{lem:flow}, we can compare the H\"older regularity of $X$ and the H\"older regularity of its projections on flows.

\begin{proposition}
Let $X=\{X_U;\;U\in\Acal\}$ be a set-indexed process on $(\mathcal{T},\mathcal{A},m)$, with finite H\"older exponents at $U_0 \in \Acal$. Then,
\begin{align*}
\inf_{f\in S^e(\Acal,U_0)} \alphar_{X^{f,m}}(t_0) = \alphar_{X,\Ccal}(U_0) \geq \alphar_X(U_0) \quad\textrm{a.s.}\\
\inf_{f\in S^e(\Acal,U_0)} \widetilde{\alphar}_{X^{f,m}}(t_0) = \widetilde{\alphar}_{X,\Ccal}(U_0) \geq \widetilde{\alphar}_X(U_0) \quad\textrm{a.s.}
\end{align*}
where the metric considered on $\Acal$ is $d_m$.
\end{proposition}

\begin{proof}
The proof is only given for the pointwise H\"older exponent. The case of the local H\"older exponent is totally similar.\\
From Proposition \ref{propComparaisonAlphas}, the inequality $\alphar_{X,\Ccal}(U_0) \geq \alphar_X(U_0)$ for all $\omega\in\Omega$ is already known.

\noindent The equality $\displaystyle\inf_{f\in S^e(\Acal,U_0)} \alphar_{X^{f,m}}(t_0) = \alphar_{X,\Ccal}(U_0)$ follows from Corollary~\ref{cor:expC} and Lemma~\ref{lem:flow}.
\end{proof}

The natural question is then to wonder if the previous inequality could be improved in an equality. The answer is generally no, as simple (deterministic) counter-examples can be built.


\section{Almost sure values for the H\"older exponents}\label{sec:as-exp}

As in the real-parameter case, we prove that the random H\"older exponents of the sample paths have almost sure values when the process is Gaussian: these values are determined in Theorems \ref{th_almostsure}  and \ref{th_unifalmostsure}.

\subsection{Uniform results for Gaussian processes}
Recall that according to Remark \ref{rem:hypGauss}, Condition (\ref{eq:assump2}) can be removed from Assumption (\HA) when the process $X$ is Gaussian and therefore in all this section.

\begin{theorem}\label{th_almostsure}
Let $X=\left\{X_U ;\; U\in\mathcal{A}\right\}$ a set-indexed centered Gaussian process, where $(\Acal_n)_{n\in\N}$ and $\dA$ satisfy Assumption (\HA).
If the deterministic local H\"older exponent of $X$ at $U_0\in\mathcal{A}$ is positive and finite, we have
\begin{align*}
\PP\big( \widetilde{\alphar}_X(U_0)=\widetilde{\mathbb{\bbalpha}}_X(U_0) \big)=1\ ,
\end{align*}
and
\begin{align*}
\PP\big( \alphar_X(U_0)=\mathbb{\bbalpha}_X(U_0) \big)=1.
\end{align*}
\end{theorem}

\

In a similar way to Theorem 3.14 of  \cite{ehjlv}, we can also obtain almost sure results on the exponents $\alphar_X(U_0)$ and $\widetilde{\alphar}_X(U_0)$ uniformly in $U_0\in\mathcal{A}$.

\begin{theorem} \label{th_unifalmostsure}
Let $X=\left\{X_U ;\; U\in\mathcal{A}\right\}$ be a set-indexed centered 
Gaussian process, where $(\Acal_n)_{n\in\N}$ and $\dA$ satisfy Assumption (\HA). \\
Suppose that the functions $U_0 \mapsto \liminf_{U\rightarrow U_0} \tilde{\mathbb{\bbalpha}}_X(U)$ and $U_0 \mapsto \liminf_{U\rightarrow U_0} \mathbb{\bbalpha}_X(U)$ are positive over $\Acal$. Then, with probability one,
\begin{equation}
\forall U_0 \in \Acal, \quad
\liminf_{U\rightarrow U_0} \widetilde{\mathbb{\bbalpha}}_X(U) \leq \widetilde{\alphar}_X(U_0) \leq \limsup_{U\rightarrow U_0} \widetilde{\mathbb{\bbalpha}}_X(U)
\end{equation}
and
\begin{equation}
\forall U_0 \in \Acal, \quad
\liminf_{U\rightarrow U_0} \mathbb{\bbalpha}_X(U) \leq \alphar_X(U_0). 
\end{equation}
\end{theorem}

\noindent The proof of Theorem \ref{th_almostsure} is an adaptation of proofs in \cite{herbin}, but with a conceptual improvement due to the well-suited formulation of Assumption (\HA), and a technical improvement in Section \ref{sec:lowerHolder} that we obtained through Theorem \ref{kolmth}. The proofs are detailed in Appendix \ref{app:2}. The proof of Theorem \ref{th_unifalmostsure} is given in Section \ref{sec:unif}.

\

\subsection{Corollaries for the $\Ccal$-H\"older exponents and the pointwise continuity exponent}

Theorem \ref{th_almostsure} can be transposed to the $\Ccal$-H\"older exponent, and the pointwise continuity exponent. 

\noindent If $X$ is a Gaussian set-indexed process, we define respectively the {\em deterministic pointwise} and {\em local $\Ccal$-H\"older exponents} on one hand, for all integer $l\geq 1$,
\begin{eqnarray*}
&&\mathbb{\bbalpha}_{X,\Ccal^l}(U_0) = \sup \bigg\{\alpha : \limsup_{\rho\rightarrow 0} \sup_{\substack{U\in B_{\dA}(U_0,\rho)\\ V\in \mathcal{B}^l(U_0,\rho)}} 
\frac{\EE[|\Delta X_{U\setminus V}|^2]}{\rho^{2 \alpha}}<\infty \bigg\}, \\
&&\widetilde{\mathbb{\bbalpha}}_{X,\Ccal^l}(U_0) = \sup \bigg\{\alpha : \limsup_{\rho\rightarrow 0} \sup_{\substack{U\in B_{\dA}(U_0,\rho)\\ V\in \mathcal{B}^l(U_0,\rho)}} 
\frac{\EE[|\Delta X_{U\setminus V}|^2]}{\dA(U,V)^{2 \alpha}}<\infty \bigg\}
\end{eqnarray*}
and the {\em deterministic pointwise continuity exponent} on the other hand,
\begin{eqnarray*}
\mathbb{\bbalpha}_X^{pc}(t_0) = \sup \left\{\alpha : \limsup_{n\rightarrow \infty} 
\frac{\EE[|\Delta X_{C_n(t)}|^2]}{m(C_n(t))^{2 \alpha}}<\infty \right\}.
\end{eqnarray*}


Similarly to Proposition \ref{propComparaisonAlphas}, the pointwise and local deterministic exponents do not depend on $l$. Hence they are denoted respectively by $\mathbb{\bbalpha}_{X,\Ccal}(U_0)$ and $\widetilde{\mathbb{\bbalpha}}_{X,\Ccal}(U_0)$.

\begin{corollary}\label{CorDeterministicCExp}
Let $X = \{X_U, \ U\in \Acal\}$ be a centered Gaussian set-indexed process. \\
If the subcollections $(\Acal_n)_{n\in\N}$ satisfy Assumption (\HA) and if the deterministic $\Ccal$-H\"older exponents are finite, then for $U_0 \in \Acal$,
\begin{align*}
\alphar_{X,\Ccal}(U_0) = \mathbb{\bbalpha}_{X,\Ccal}(U_0) \ \textrm{a.s.}
\quad\textrm{and}\quad
\widetilde{\alphar}_{X,\Ccal}(U_0) = \widetilde{\mathbb{\bbalpha}}_{X,\Ccal}(U_0) \ \textrm{a.s.}
\end{align*}
\end{corollary}

\begin{proof}
It suffices to prove the result for $l=1$, which corresponds to $V\subseteq U$ in the definition of the standard H\"older exponent. Thus one can apply the previous proofs (Sections \ref{sec:lowerHolder} and \ref{sec:upperHolder}) which are still valid when restricted to $V\subseteq U$.
\end{proof}


\begin{corollary}\label{CorDeterministicExppc}
Let $X = \{X_U, \ U\in \Acal\}$ be a centered Gaussian set-indexed process. If the deterministic exponent of pointwise continuity is finite, then for $t_0\in \mathcal{T}$,
\begin{equation*}
\alphar_X^{pc}(t_0) = \mathbb{\bbalpha}_X^{pc}(t_0) \quad\textrm{a.s.}
\end{equation*}
Moreover, for any $U_{max} \in \Acal$ such that $m(U_{max})<\infty$,
\begin{equation*}
\PP\left( \forall t\in U_{max}, \ \alphar_X^{pc}(t) \geq \mathbb{\bbalpha}_X^{pc}(t)  \right)=1.
\end{equation*}
\end{corollary}

\begin{proof}
Fix $t_0 \in\mathcal{T}$. Let $\alpha< \mathbb{\bbalpha}_X^{pc}(t_0)$. 
The inequality $\alpha< \alphar_X^{pc}(t_0)$ a.s. is a direct consequence of Proposition \ref{prop:crit-pc-unif}. This gives $\alphar_X^{pc}(t_0) \geq \mathbb{\bbalpha}_X^{pc}(t_0)$ almost surely.


For the converse inequality, denote $\mu = \mathbb{\bbalpha}_X^{pc}(t_0)$. Then for all $\epsilon>0$, there exist a subsequence $\left(C_{\varphi(n)}(t_0)\right)_{n\in\N}$ of $\left(C_n(t_0)\right)_{n\in\N}$ and a constant $c>0$ such that
\begin{equation*}
\forall n \in \N^*, \quad 
\EE\big[|\Delta X_{C_{\varphi(n)}(t_0)}|^2\big] 
\geq c \ m(C_{\varphi(n)}(t_0))^{2\mu + \epsilon}.
\end{equation*}
For all $n\in\N$, the law of the random variable $\displaystyle\frac{\Delta X_{C_{\varphi(n)}(t_0)}}{m(C_{\varphi(n)}(t_0))^{\mu + \epsilon}}$ is $\mathcal{N}(0, \sigma_n^2)$. The previous inequality implies that $\sigma_n \rightarrow \infty$ as $n\rightarrow\infty$. 
Then for all $\lambda>0$, the same computation as in Lemmas \ref{lem_uppoint} and \ref{lem_upbound} leads to
\begin{align*}
\PP\left( \frac{m(C_{\varphi(n)}(t_0))^{\mu + \epsilon}}{\Delta X_{C_{\varphi(n)}(t_0)}}<\lambda \right)
&=\PP\left( \frac{\Delta X_{C_{\varphi(n)}(t_0)}}{m(C_{\varphi(n)}(t_0))^{\mu + \epsilon}} > \frac{1}{\lambda} \right)\\
&=\int_{|x|>\frac{1}{\lambda}}\frac{1}{\sqrt{2\pi}\sigma_n}
\exp\left(-\frac{x^2}{2\sigma_n^2}\right).dx\\
&=\frac{1}{2\pi}\int_{|x|>\frac{1}{\lambda\sigma_n}}
\exp\left(-\frac{x^2}{2}\right).dx\stackrel{n\rightarrow +\infty}
{\longrightarrow} 1.
\end{align*}
Therefore the sequence 
$\displaystyle\left(\frac{m(C_{\varphi(n)}(t_0))^{\mu + \epsilon}}{\Delta X_{C_{\varphi(n)}(t_0)}}\right)_{n\in\N}$
converges to $0$ in probability. As a consequence, there exists a subsequence which converges
to $0$ almost surely.
Then for all $\epsilon>0$, we have almost surely $\alphar_X^{pc}(t_0)\leq\mu+\epsilon$. Taking $\epsilon\in\Q_{+}$, this yields $\alphar_X^{pc}(t_0) \leq \mathbb{\bbalpha}_X^{pc}(t_0)$ a.s.


\noindent The second equation is a direct consequence of Proposition \ref{prop:crit-pc-unif}.
\end{proof}


\section{Application: H\"older regularity of the set-indexed fractional Brownian motion} \label{sec:appl}

The general results proved in Section \ref{sec:as-exp} allow to describe the local behaviour of the recently defined set-indexed extensions of fractional Brownian motion. In fact, the L\'evy fractional Brownian motion and the fractional Brownian sheet can be considered as set-indexed processes, and SI fractional Brownian motions recently appeared as limit processes in functional Central Limit Theorems \cite{Bierme}. 

\subsection{H\"older exponents of the SIfBm}

The local regularity of fractional Brownian motion $B^H=\{B^H_t;\;t\in\R_+\}$ is known to be constant a.s. and given by the self-similarity index $H\in (0,1)$. 
The two classical H\"older exponents satisfy, a.s., 
\[
\forall t\in\R_+,\quad\alphar_{B^H}(t)=\widetilde{\alphar}_{B^H}(t)=H.
\]

In \cite{ehem, ehem2}, a set-indexed extension for fractional Brownian motion has been defined and studied. 
A mean-zero Gaussian process $\sifbm^H = \left\{\sifbm_U^H, U\in \mathcal{A} \right\}$ is called a {\em set-indexed fractional Brownian motion (SIfBm for short)} on $(\mathcal{T},\mathcal{A},m)$ if
\begin{equation}
\forall U,V \in\mathcal{A},\quad
\EE\left[\sifbm_U^H \sifbm_V^H \right] = \frac{1}{2} \left[ m(U)^{2H} + m(V)^{2H} - m(U\bigtriangleup V)^{2H} \right],
\end{equation}
where $H\in (0,1/2]$ is the index of self-similarity of the process.

In \cite{hausdorff}, the deterministic local H\"older exponent and the almost sure value of the local H\"older exponent have been determined for the particular case of an SIfBm indexed by the collection $\{[0,t];\;t\in\R_+^N\}\cup\{\emptyset\}$, called the {\em multiparameter fractional Brownian motion}.
If $X$ denotes the $\R^N_+$-indexed process defined by $X_t=\sifbm^H_{[0,t]}$ for all $t\in\R_+^N$, then for all $t_0\in\R_+^N$, $\widetilde{\mathbb{\bbalpha}}_X(t_0) = H$ and with probability one, for all $t_0\in\R_+^N$, $\widetilde{\alphar}_X(t_0) = H$.

However, the local regularity has not been studied so far, in the general case of an indexing collection which is not reduced to the collection of rectangles of $\R^N_+$.
Theorem~\ref{th_almostsure}, \ref{th_unifalmostsure} and Corollary~\ref{CorDeterministicExppc} provide new results for the sample paths of SIfBm.

In Section \ref{sec:as-exp}, Theorem~\ref{th_unifalmostsure} failed to provide a uniform almost sure upper bound for the pointwise H\"older exponent of a general Gaussian set-indexed process. 
In the specific case of the set-indexed fractional Brownian motion, this result can be improved under the following additional requirement: assume there exists $\eta>0$ such that $\forall U_0 \in \Acal$,
\begin{equation}\label{eqHypFin}
\inf_{\rho > 0} \sup\left\{\frac{d_m(U,g_n(U))}{\rho}; \ n\in \N,\ U, g_n(U) \in B_{d_m}(U_0,\rho) \right\}  \geq \eta \ .
\end{equation}


\begin{theorem}\label{th_hoSifBm}
Let $\sifbm^H$ be a set-indexed fractional Brownian motion on $\left(\mathcal{T},\Acal,m \right)$, $H\in (0,1/2]$. Assume that the subclasses $(\Acal_n)_{n\in\N}$ satisfy Assumption (\HA). \\
Then, the local and pointwise H\"older exponents of $\sifbm^H$ at any $U_0\in\mathcal{A}$, defined with respect to the distance $d_m$ or any equivalent distance, satisfy 
\begin{equation*}
\PP \left( \forall U_0\in \mathcal{A},\ \widetilde{\alphar}_{\sifbm^H}(U_0)=H \right) = 1
\end{equation*}
and if the additional Condition (\ref{eqHypFin}) holds,
\begin{equation*}
\PP \left( \forall U_0\in \mathcal{A},\ \alphar_{\sifbm^H}(U_0) = H \right) = 1.
\end{equation*}
In particular, this holds for the multiparameter fBm since (\ref{eqHypFin}) is true in $\R_+^N$.
\end{theorem}

\begin{proof}
The following expression of the incremental variance,
\begin{equation*}
\forall U,V\in\mathcal{A},\quad
\EE\left[|\sifbm_U^H -\sifbm_V^H |^2\right] = m(U\bigtriangleup V)^{2H},
\end{equation*}
directly implies that the deterministic pointwise and local H\"older exponents are equal to $H$. By Theorem \ref{th_almostsure}, the random exponents on an indexing collection satisfying Assumption (\HA) are also equal to $H$.

\noindent For the uniform almost sure result on $\Acal$, according to Theorem~\ref{th_unifalmostsure}, it remains to prove that
$\PP \left( \forall U_0\in \mathcal{A},\ \alphar_{\sifbm^H}(U_0) \leq H \right) = 1.$
This is the object of the following Section~\ref{sec:unif-as-pointwise-sifbm}.

For the particular case of the multiparameter fBm, it suffices to notice that the collection $\Acal$ of rectangles of $\R_+^N$ endowed with the Lebesgue measure $\lambda$ satisfies Condition (\ref{eqHypFin}). 
In that case, let us recall that for any $U_0\in \Acal$, $d_{\lambda}(U_0, g_n(U_0))=N.2^{-n}+o(2^{-n})$. Hence for a given $\rho>0$, choosing the smallest integer $n$ such that $N.2^{-n} \leq \rho/2$ ensures that
\begin{equation*}
\frac{d_{\lambda}(U_0,g_n(U_0))}{\rho} \geq \frac{N.2^{-(n+1)}}{\rho}\geq \frac{1}{8},
\end{equation*}
and that $g_n(U_0)\in B_{d_{\lambda}}(U_0,\rho)$.
\end{proof}

If the collection $\Acal$ or the metric $d_m$ do not satisfy the additional requirement (\ref{eqHypFin}), then the lower bound for the pointwise exponent remains true by Theorem~\ref{th_unifalmostsure}:
$\PP \left( \forall U_0\in \mathcal{A},\ \alphar_{\sifbm^H}(U_0) \geq H \right) = 1$.

\

In \cite{ehem}, it is shown that for all $U,V \in \Acal$, $\EE[|\Delta \sifbm^H_{U\setminus V} |^2] = m(U\setminus V)^{2H}$. This implies that for all $U_0 \in \Acal$, $\widetilde{\mathbb{\bbalpha}}_{\sifbm^H,\Ccal}(U_0) = \mathbb{\bbalpha}_{\sifbm^H,\Ccal}(U_0)=H$, and so by Corollary \ref{CorDeterministicCExp}:
\begin{equation*}
\widetilde{\alphar}_{\sifbm^H,\Ccal}(U_0) = \alphar_{\sifbm^H,\Ccal}(U_0)=H \quad\textrm{a.s.}
\end{equation*}


For the pointwise continuity, one needs to determine the behaviour of $\EE\left[|\Delta\sifbm^H_C|^2\right]$ when $C\in\mathcal{C}$ (and not only $C=U\setminus V\in\mathcal{C}_0$, with $U,V\in\Acal$ as previously). In the specific case of an SIfBm with $H=1/2$, we can state:

\begin{proposition}\label{prop:BMpc}
Let $\sifbm$ be a Brownian motion on $\Acal$. Then, for all $t_0\in \mathcal{T}$,
\begin{equation*}
\alphar_{\sifbm}^{pc}(t_0) = \mathbb{\bbalpha}_{\sifbm}^{pc}(t_0) = \frac{1}{2}
\quad\textrm{a.s.}
\end{equation*}
A uniform lower bound in any $U_{max} \in \Acal$ such that $m(U_{max})<\infty$, is given by:
\begin{equation*}
\PP\left( \forall t_0 \in U_{max}, \quad 
\alphar_{\sifbm}^{pc}(t_0) \geq \mathbb{\bbalpha}_{\sifbm}^{pc}(t_0) = \frac{1}{2}
\right) = 1.
\end{equation*}
\end{proposition}

\begin{proof}
Since $\EE\left[|\Delta \sifbm_C|^2 \right] = m(C)$, the result follows from Corollary~\ref{CorDeterministicExppc}.
\end{proof}

This property cannot be extended directly to any SIfBm for which $H<1/2$, since we do not have $\EE\left[|\Delta \sifbm^H_C |^2\right] = m(C)^{2H}$ for all $C\in \Ccal$ (see \cite{ehem}). 
However, the results of the previous Proposition \ref{prop:BMpc} hold for the multiparameter fBm, as we shall see after the following technical lemma.


\begin{lemma}
\label{lemma_C_n}
Let $\Acal = \left\{ [0,t]: t\in [0,1]^N \right\}$ endowed with the usual dissecting class $(\Acal_n)$ made of the dyadics. Let $t\in (0,1)^N$, $t=(t_1,\dots,t_N)$ and define: 
\begin{equation*}
t_j^n = \left\{
    \begin{array}{ll}
        t_j & \textrm{if \ } 2^n t_j\in \N\\
        2^{-n}\lfloor 2^n t_j+1 \rfloor & \mbox{otherwise,}
    \end{array}
\right.
\textrm{ and }\quad
\tilde{t}_k^n = \left\{
    \begin{array}{ll}
        2^{-n}\lfloor 2^n t_k -1\rfloor & \textrm{if \ } 2^n t_k\in \N \\
        2^{-n}\lfloor 2^n t_k \rfloor & \mbox{otherwise.}
    \end{array}
\right.
\end{equation*}

Then, 
\begin{equation*}
C_n(t) = [0,(t_1^n,\dots,t_N^n)]\setminus \bigcup_{k=1}^N [0,(t_1^n,\dots,\tilde{t}_k^n,\dots,t_N^n)] . 
\end{equation*}

\end{lemma}

\begin{proof}
We recall that $C_n(t)$, the left-neighbourhood of $A_t$ in $\Acal_n$, is defined as $\bigcap_{\substack{C\in \Ccal_n \\ t\in C}} C$. In the particular case of the rectangles, it corresponds to the expression given in the lemma.
\end{proof}

As usual, let $\lambda$ be the Lebesgue measure of $\R^N$. A direct consequence of this result is that any Gaussian process $X$ satisfying the assumptions of Corollary \ref{GaussKolm} satisfies, for all $t\in [0,1]^N$ and for all $\omega$,
\begin{equation*}
\widetilde{\alphar}_{X,\Ccal} (A_t) \leq \alphar_{X,\Ccal} (A_t) \leq \alphar_X^{pc}(t), 
\end{equation*}
with respect to the Lebesgue measure $\lambda$ and the distance $d_{\lambda}$.

\begin{proposition}\label{prop:SIfBmpc}
Let $\left\{ \sifbm^H_U: \ U\in \Acal\right\}$ be a SIfBm process, where $\Acal$ refers to the rectangles of $[0,1]^N$. 
Then, the pointwise continuity of $\sifbm^H$ with respect to the Lebesgue measure $\lambda$ of $\R^N$ satisfies 
\begin{equation*}
\forall t_0\in [0,1]^N, \quad \PP\left(\alphar_{\sifbm^H}^{pc}(t_0) = \mathbb{\bbalpha}_{\sifbm^H}^{pc}(t_0)= H \right)=1, 
\end{equation*}
and,
\begin{equation*}
\PP\left(\forall t_0\in [0,1]^N, \quad \alphar_{\sifbm^H}^{pc}(t_0) \geq \mathbb{\bbalpha}_{\sifbm^H}^{pc}(t_0)= H \right) = 1.
\end{equation*}
\end{proposition}

\begin{proof}
For the sake of readability, the proof is written for $N=2$.
Let $t=(t_1,t_2) \in [0,1]^N$. To show there is no difference in the final result, we assume $t_1$ is dyadic and $t_2$ is not. Let $k,l\in \N, k<2^l$ such that $t_1=k.2^{-l}$. Let $n\in \N, n\geq l$.

\noindent First, we notice that, by Lemma \ref{lemma_C_n},
\begin{equation*}
C_n(t) = \left[0,(t_1,2^{-n}\lfloor 2^n t_2+1\rfloor)\right] \setminus \left\{ \left[0,(2^{-n}\lfloor 2^n t_1-1 \rfloor,2^{-n}\lfloor 2^n t_2+1\rfloor)\right] \cup \left[0,(t_1,2^{-n}\lfloor 2^n t_2\rfloor)\right]\right\}.
\end{equation*}
Re-writing this for short $C_n(t) = A_n \setminus \left\{ B_{1,n} \cup B_{2,n}\right\}$, the inclusion-exclusion formula gives
\begin{align*}
\EE\left[| \Delta \sifbm^H_{C_n(t)} |^2\right]  &= m(A_n\setminus B_{1,n})^{2H} + m(A_n\setminus B_{2,n})^{2H} -  m(A_n\setminus (B_{1,n}\cap B_{2,n})^{2H} \\
&\quad - m(B_{1,n}\bigtriangleup B_{2,n})^{2H} +  m(B_{1,n}\setminus B_{2,n})^{2H} + m(B_{2,n}\setminus B_{1,n})^{2H}. 
\end{align*}
A simple estimation of all the previous terms then gives $\mathbb{\bbalpha}_{\sifbm^H}^{pc}(t_0)=H$ for all $t_0\in [0,1]^N$, and the second assertion follows by Corollary \ref{CorDeterministicExppc}.
\end{proof}


\subsection{Proof of the uniform a.s. pointwise exponent of the SIfBm}
\label{sec:unif-as-pointwise-sifbm}

In \cite{BH2}, a similar result is proved for the regular multifractional Brownian motion, with a proof based on the integral representation of this process.

We shall adapt the following technical lemma taken from \cite{BH2}:
\begin{lemma}\label{lemma_fBm}
Let $B^H =\{B^H_t, t\in \R_+\}$ be a fractional Brownian motion of index $H\in (0,1)$. Let $\epsilon>0$, $\rho>0$, $0\leq s<t$, $n\in \N^*$ and $\delta u = \frac{\rho}{n}$. Then, let $u_0 = s$ and for all $k\in \{0,\dots,n\}$, $u_{k+1} = u_k + \delta u$. We have the following:
\begin{equation*}
\PP\left(\bigcap_{k=1}^n \{|B^H_{u_k}-B^H_{u_{k-1}}|<\rho^{H+\epsilon}\}\right) \leq \left(\frac{2}{\sqrt{2\pi}}\right)^n \left(\frac{\rho^{H+\epsilon}}{C.(\delta u)^H}\right)^n \ ,
\end{equation*} 
where $C$ is a constant depending only on $H$.
\end{lemma}

In the sequel, for $U\subset V \in \Acal$, we denote by $\mathcal{R}(f,U\rightarrow V)$, the range of the elementary flow $f:[0,d]\rightarrow \Acal$ such that $f(0)=U$ and $f(d)=V$, where $d=d_m(U,V)$ (the distance considered here is always $d_m=m(\bullet\bigtriangleup\bullet)$). Hence $\mathcal{R}(f,U\rightarrow V)$ is a totally ordered subset of $\Acal$ which forms a continuum. 
We also denote by $\mathcal{R}_n(f,U)$, the range $\mathcal{R}(f,U\rightarrow g_n(U))$. Since the choice of a particular $f$ does not matter, these notations can be used without specifying $f$, considering that a choice has been made. 

\begin{lemma}\label{lemma_Sifbm}
Let $\sifbm^H$ be a SIfBm on $(\Acal,\mathcal{T},m)$ of index $H\in (0,\frac{1}{2}]$. Let $U\in \Acal$, $i\in \N$ and $\rho_i=d_{m}(U,g_i(U))$. 
Let $\epsilon>0$, $n\in \N^*$. 
In any $\mathcal{R}_i(f,U)$, there exist an increasing sequence $(U_j)_{0\leq j\leq n}$ such that $U_0=U$, $U_n=g_i(U)$, and $\delta U = d_{m}(U_{j-1},U_j) = \frac{\rho_i}{n}$ for all $j\in \{1,\dots,n\}$. 
Then,
\begin{equation}\label{eqSIfBm}
\PP\left( \bigcap_{k=1}^n \left\{|\sifbm^H_{U_k}-\sifbm^H_{U_{k-1}}|<\rho_i^{H+\epsilon} \right\} \right)\leq \left( \frac{2}{\sqrt{2\pi}} \right)^n \left( \frac{\rho_i^{H+\epsilon}}{\sigma} \right)^n \leq \left(\widetilde{C}\ n^H \rho_i^{\epsilon}\right)^n ,
\end{equation}
where $\sigma = C. (\delta U)^{H}$ and $C,\widetilde{C}>0$ only depends on $H$. 
\end{lemma}

\begin{proof}
Let us consider the range $\mathcal{R}_i(f,U)$ of a flow $f$ connecting $U$ to $g_i(U)$. The standard projection of $X=\sifbm^H$ on $f$ is a standard fractional Brownian motion that we denote $X^{f,m}=\left\{X^{f,m}_t, t\in [0,\rho_i]\right\}$. As usual, $\theta = m\circ f$ and in the present situation, $\theta:[0,\rho_i] \rightarrow [m(U), m(g_i(U))]$. 
For $k\in \{0,\dots,n\}$, let $u_k=m(U)+k.\frac{\rho_i}{n}$ and define $U_k=f\circ \theta^{-1}(u_k)$. The $U_k$'s contitute the sequence of the statement and we remark that
\begin{equation*}
\PP\left(\bigcap_{k=1}^p \left\{|X_{U_k}-X_{U_{k-1}}|<\rho_n^{H+\epsilon} \right\} \right) = \PP\left(\bigcap_{k=1}^p \left\{|X^{f,m}_{u_k}-X^{f,m}_{u_{k-1}}|<\rho_n^{H+\epsilon} \right\} \right).
\end{equation*}
The result follows from Lemma \ref{lemma_fBm}.
\end{proof}

The following Proposition is the last key result to prove the uniform almost sure upper bound for the SIfBm. 

\begin{proposition}\label{prop:minunif_sifbm}
Let $\sifbm^H$ be a SIfBm on $(\Acal,\mathcal{T},m)$ of parameter $H\in (0,1/2]$. 
We assume that $(\Acal_n)_{n\in\N}$, endowed with $d_m$, satisfies Assumption (\HA) and (\ref{eqHypFin}).\\
Then, for all $\epsilon>0$, there exists a finite random variable $h>0$ such that almost surely, for all $\rho \leq h(\omega)$ and for all $U_0\in \Acal$,
\begin{equation*}
\sup_{U,V\in B_{d_{\Acal}}(U_0,\rho)} \left|\sifbm^H_U - \sifbm^H_V\right| \geq \rho^{H+\epsilon}.
\end{equation*}
\end{proposition}

\begin{proof}
Let us fix $\epsilon>0$. For all $U\in \Acal$, let $\rho_{n,U} = d_m(U,g_n(U))$ and $p_{n,U}=\lfloor \rho_{n,U}^{-\epsilon}\rfloor$. 
For all $N\in \N^*$, we consider the event
\begin{align*}
A_N = \bigcup_{n\geq N} \bigcup_{U\in \Acal_n} \left\{ \forall V,W\in \mathcal{R}_n(f,U), |X_V-X_W|<\rho_{n,U}^{H+\epsilon} \right\}.
\end{align*}
We have
\begin{align*}
\PP(A_N)
&\leq \sum_{n\geq N} \sum_{U\in \Acal_n} \PP\left( \forall V,W\in \mathcal{R}_n(f,U), |X_V-X_W|<\rho_{n,U}^{H+\epsilon}\right) \\
&\leq \sum_{n\geq N} \sum_{U\in \Acal_n} \PP\left(\bigcap_{k=1}^{p_{n,U}} \left\{|X_{U_k}-X_{U_{k-1}}|<\rho_{n,U}^{H+\epsilon} \right\} \right), 
\end{align*}
where $U_0, \dots, U_{p_{n,U}}$ are defined as in Lemma \ref{lemma_Sifbm}.

\noindent Following equation (\ref{eqSIfBm}) and since $\rho_{n,U} = \dA(U,g_n(U)) \leq k_n^{-1/\qA}$, there exist positive constants $C_1$ and $C_2$ such that
\begin{align*}
\PP\left(\bigcap_{k=1}^{p_{n,U}} \left\{|X_{U_k}-X_{U_{k-1}}|<\rho_{n,U}^{H+\epsilon} \right\} \right) 
&\leq \left(C_1\ \rho_{n,U}^{\epsilon (1-H)}\right)^{\rho_{n,U}^{-\epsilon}-1} \\
&\leq \left(C_2\ k_n^{-1/\qA}\right)^{\epsilon (1-H) ( k_n^{\epsilon/\qA}-1)}.
\end{align*}

\noindent Going back to the previous equation, we obtain
\begin{align*}
\PP(A_N)
\leq \sum_{n\geq N} k_n \left(C_2\ k_n^{-1/\qA}\right)^{\epsilon (1-H)(k_n^{\epsilon/\qA}-1)} = R_N.
\end{align*}

\noindent Since $k_n$ is admissible, we can easily show that $\sum_{N\in \N^*} R_N <\infty$. Hence, Borel-Cantelli Lemma implies the existence of a random variable $N(\omega)$ such that: with probability one, for all $n\geq N(\omega)$ and for all $U\in \Acal_n$, 
\begin{equation}\label{eqSIunif}
\exists V, W\in \mathcal{R}_n(f,U) ; \quad |X_V-X_W| \geq \rho_{n,U}^{H+\epsilon}.
\end{equation}
For $U_0\in \Acal$ and $\rho>0$, Assumption (\ref{eqHypFin}) gives the existence of $\mathcal{R}_n(f,U)\subset B_{\dA}(U_0,\rho)$, for some $n\geq N(\omega)$ and $U\in\Acal$ such that $\rho_{n,U} \geq \eta \rho$. Then, there exist $V,W\in\Acal$ (the same that in (\ref{eqSIunif})), such that
\begin{equation*}
|X_V-X_W|\geq \rho_{n,U}^{H+\epsilon} \geq (\eta^{H+\epsilon})\ \rho^{H+\epsilon}
\end{equation*}
which concludes the proof.
\end{proof}


\appendixpage
\begin{appendices}

\section{Complements on Examples \ref{ex:non-integer exponent} and \ref{ex:Hfsm}}\label{app:0}

In this appendix, we use directly the notations of the aforementionned examples without recalling them.

\noindent {\bf Example \ref{ex:non-integer exponent}.} We complete this example with an estimate of $J_n$ and $k_n$, and conclude with the computations that permit to obtain the value of $\qA$.
Since on any dyadic cube of size $2^{-n}$, $X$ has increments of order $2^{-n(H-\epsilon)}$ (we then omit $\epsilon$), we deduce that $J_n$ has to be of order $2^{-nH}/2^{-n}$, and then $k_n$ has to be of order $2^{n(N+1-H)}$. We refer to \cite{XiaoHausdorff} for rigorous proofs of these facts.\\
For $\varepsilon>0$ small (compared to $2^{-n}$), an $\varepsilon$-covering provides the following bound:
\begin{align*}
\mu_\varphi&\left(\{(s,X_s(\omega)),\ s\in [0,2^{-n}(k+1)]\setminus [0,2^{-n}k],\ j< 2^n X_s(\omega)\leq j+1\}\right)\\
&\quad \quad \quad\quad \quad \quad \quad \quad \quad \leq \varphi(\varepsilon)\times \#\{\textrm{$\varepsilon$-cubes to cover the previous set}\}\\
& \quad \quad \quad\quad \quad \quad \quad \quad \quad \leq C\ \varphi(\varepsilon)\times \underbrace{N\frac{2^{-n}}{\varepsilon^N}}_{\varepsilon-\text{covering of }[0,2^{-n}(k+1)]\setminus [0,2^{-n}k]}\ \times \underbrace{\frac{\varepsilon^H}{\varepsilon}}_{\#\{\text{$\varepsilon$-intervals for increment of size $\varepsilon^H$}\}} \\
&\quad \quad \quad\quad \quad \quad \quad \quad \quad \leq C\ 2^{-n} =  C\ k_n^{-1/(N+1-H)} \ ,
\end{align*}
for some constant $C>0$, so that we can identify $\qA$ with $N+1-H$.

\noindent {\bf Example \ref{ex:Hfsm}.} In \cite{AR14}, it is explained how to construct a stable measure on an abstract Hilbert space $E$ such that:
\begin{equation*}
\|S\xi\|_{L^2(\mathcal{T},m)}^{\alpha H} = 2 \int_E (1-\cos \langle\xi,x\rangle)\ \Delta^{\alpha H}(dx) \ ,
\end{equation*}
where $\Delta^{\alpha H}$ is the L\'evy measure of an $\alpha H$-stable random measure on $E$ and $S$ is the canonical embedding from $E^*$ to $L^2(\mathcal{T},m)$.\\
If $M^\alpha$ denotes an $\alpha$-stable random measure on $E$ (see \cite{Samorodnitsky} for definition and properties of stable random measures and integrals) with control measure $\Delta^{\alpha H}$, let us define the following process on $E^*$:
\begin{equation*}
\tilde{X}_\xi = \int_E \left(1-e^{i\langle \xi,x\rangle}\right)\ dM^\alpha_x \ .
\end{equation*}
Then, the scale parameter of $\tilde{X}_\xi - \tilde{X}_\eta$ is given by:
\begin{align*}
\|\tilde{X}_\xi - \tilde{X}_\eta\|_\alpha^\alpha \quad &= \int_E \left|1-e^{i\langle \xi,x\rangle} - (1-e^{i\langle \eta,x\rangle})\right|^{\alpha}\ \Delta^{\alpha H}(dx)\\
&= \int_E \left(2-2 \cos\langle \xi-\eta,x\rangle \right)^{\frac{\alpha}{2}}\ \Delta^{\alpha H}(dx) \ .
\end{align*}
It is known (see \cite{AR14} for a list of references) that $\Delta^{\alpha H}$ has a radial decomposition, just as in the real case. Hence, let $\sigma$ be the measure on the unit sphere $\mathcal{S}$ of $E$ such that:
\begin{align*}
\|\tilde{X}_\xi - \tilde{X}_\eta\|_\alpha^\alpha \quad &= \int_0^\infty \int_{\mathcal{S}} \left(2-2 \cos\langle \xi-\eta,r y\rangle \right)^{\frac{\alpha}{2}}\ \frac{dr}{r^{1+\alpha H}}\ \sigma(dy) \ ,
\end{align*}
from which we deduce that the above quantity scales as $\|S(\xi-\eta)\|^{\alpha H}_{L^2(\mathcal{T},m)}$. We do not discuss here how the previous equation extends isometrically from $E^*$ to $L^2(\mathcal{T},m)$, and this finally implies (\cite[p.18]{Samorodnitsky}) that,
for any $\gamma<\alpha$, $f,g\in L^2(\mathcal{T},m)$,
\begin{equation*}
\EE\left(|\tilde{X}_f - \tilde{X}_g|^\gamma\right) \leq C\ \|f-g\|^{\gamma H}_{L^2(\mathcal{T},m)} \ ,
\end{equation*}
for some $C>0$ that depends on $\alpha$ and $\gamma$, but not $f$ and $g$. $X$ is then defined by $X_U = \tilde{X}_{\mathbf{1}_U},\ U\in \Acal$.

\section{Proof of Corollary \ref{corCl}}\label{app:1}

In order to prove Corollary \ref{corCl}, we need the following lemma:
\begin{lemma}\label{lem:Cl}
If the distance $\dA$ on the class $\mathcal{A}$ is contracting, then for any integer $l\geq 1$ and for any $U, V_1,\dots, V_l \in \Acal$,
\begin{equation*}
\max_{i\leq l} \{\dA(U,V_i)\}\leq \rho  \Rightarrow 
\dA(U,V_1\cap \dots \cap V_l)\leq K(l)\ \rho,
\end{equation*}
for some constant $K(l)>0$ which only depends on $l$.
\end{lemma}

\begin{proof}[Proof of Lemma \ref{lem:Cl}]
The proof relies on the triangular inequality and the contracting property of $\dA$.
\end{proof}

\begin{proof}[Proof of Corollary \ref{corCl}]
$g_n$ can be extended to $\mathcal{A}(u)$ in the following way:
\begin{equation*}
\forall V_1,\dots,V_p \in \mathcal{A}, \quad 
g_n\bigg(\bigcup_{i=1}^p V_i\bigg) = \bigcup_{i=1}^p g_n(V_i) \ ,
\end{equation*}

\noindent so the following inequality holds:
\begin{align}\label{eq:maj-deltaXC}
|X_U - \Delta X_{\cup V_i}| \leq |X_{g_{n_0}(U)} - \Delta X_{g_{n_0}(\cup V_i)}| +& \sum_{j\geq n_0}|X_{g_{j+1}(U)} - X_{g_j(U)}| \nonumber\\
&+ \sum_{j\geq n_0}|\Delta X_{g_{j+1}(\cup V_i)} - \Delta X_{g_j(\cup V_i)}| \ .
\end{align}

Then by the inclusion-exclusion formula,
\begin{align}\label{eq:delta-gn}
|\Delta X_{g_{n+1}(\cup V_i)} - \Delta X_{g_n(\cup V_i)}| \leq& \sum_{i=1}^p |X_{g_{n+1}(V_i)} - X_{g_n(V_i)}| +\dots  \nonumber\\
&+ \sum_{i_1<\dots<i_k} |X_{g_{n+1}(\cap_{i_1<\dots<i_k} V_i)} - X_{g_n(\cap_{i_1<\dots<i_k} V_i)}| +\dots \nonumber\\
&+ |X_{g_{n+1}(\cap_{i=1}^p V_i)} - X_{g_n(\cap_{i=1}^p V_i)}|.
\end{align}

Now assume that $U,V_1,\dots, V_p \in \mathcal{D}$. When $p\leq l$, the number of terms in the right side of inequality (\ref{eq:delta-gn}) is bounded by a constant, independent of the sets $V_1,\dots,V_p\in\mathcal{A}$.
Thus, there exists a positive constant $K_2(l)$ such that
\begin{equation}\label{eqChain4}
|\Delta X_{g_{n+1}(\cup V_i)} - \Delta X_{g_n(\cup V_i)}| 
\leq K_2(l) \ \sup_{W\in \mathcal{D}} |X_{g_{n+1}(W)} - X_{g_n(W)}| \ .
\end{equation}

\noindent Using the same sequence $(a_j)_{j\in \N}$ as in the proof of Theorem \ref{kolmth}, and the above equation (\ref{eqChain4}) in the third inequality below:
\begin{align*}
\mathbf{P}\bigg(\sup_{V_1,\dots,V_p \in \mathcal{D}}&\sum_{j\geq n_0}|\Delta X_{g_{j+1}(\cup V_i)} - \Delta X_{g_j(\cup V_i)}| \geq K_2(l) k_{n_0+1}^{-\gamma/\qA}\bigg) \\ &\leq \PP\left(\exists V_1,\dots, V_p\in \mathcal{D}, \exists j\geq n_0, \ |\Delta X_{g_{j+1}(\cup V_i)} - \Delta X_{g_j(\cup V_i)}| \geq a_j K_2(l) k_{n_0+1}^{-\gamma/\qA}\right) \\
&\leq \PP\left(\exists W\in \mathcal{D}, \exists j\geq n_0, \ |X_{g_{j+1}(W)} -  X_{g_j(W)}| \geq a_j  k_{n_0+1}^{-\gamma/\qA}\right) \ .
\end{align*}
We obtain the same expression (\ref{eq:proofKolm1}) that we had in the proof of Theorem \ref{kolmth}, thus the same conclusion holds: if $\max_{i\leq l}\left\{m(U\setminus V_i)\right\}\leq k_{n_0}^{-1/\qA}$, then almost surely, ${k_{n_0}^{-1/\qA} \leq h^*}$ implies that:
\begin{align*}
\sup_{V_1,\dots,V_p \in \mathcal{D}}&\sum_{j\geq n_0}|\Delta X_{g_{j+1}(\cup V_i)} - \Delta X_{g_j(\cup V_i)}| \leq K_2(l)\ k_{n_0+1}^{-\gamma/\qA} \ .
\end{align*}

In the same way, the second term of the upper bound (\ref{eq:maj-deltaXC}) is proved to be bounded by $K_4(\gamma,\qA) \ m(C)^{\gamma}$, where $K_4(\gamma,\qA)>0$ only depends on $\gamma$ and $\qA$.\\
The first term of (\ref{eq:maj-deltaXC}) can be bounded by a finite sum (whose number of terms only depends on $l$) of the form $|X_{g_{n_0}(U)} - X_{g_{n_0}(V_{i_1,\dots,i_k})}|$, where $V_{i_1,\dots,i_k} = V_{i_1} \cap \dots \cap V_{i_k}$ for $i_1<\dots<i_k\leq l$:
\begin{equation}\label{eqC}
|X_{g_{n_0}(U)}-\Delta X_{g_{n_0}(\cup V_i)}| \leq \sum_{j=1}^l \sum_{i_1<\dots<i_j} |X_{g_{n_0}(U)}-X_{g_{n_0}(V_{i_1,\dots,i_j})}| .
\end{equation}
Finally, if $\max_{i\leq l}\left\{m(U\setminus V_i)\right\}\leq k_{n_0}^{-1/\qA}$, condition (\ref{eq:assump1}) of Assumption (\HA) and Lemma \ref{lem:Cl} imply 
\begin{align*}
d_m(g_{n_0}(U),g_{n_0}(V_{i_1,\dots,i_j})) 
&\leq d_m(g_{n_0}(U), U) + d_m(U,V_{i_1,\dots,i_j}) + d_m(V_{i_1,\dots,i_j},g_{n_0}(V_{i_1,\dots,i_j}))\\
&\leq K(l)\ \max_{i\leq l}\left\{m(U\setminus V_i)\right\} + 2 k_{n_0}^{-1/\qA} \\
&\leq (K(l)+2) \ k_{n_0}^{-1/\qA}\ . 
\end{align*}
Hence, Theorem \ref{kolmth} implies that when $k_{n_0}^{-1/\qA}<(K(l)+2)^{-1}\ h^*$, each term of equation (\ref{eqC}) is bounded by a quantity proportional to $m(C)^{\gamma}$. Then, the random variable $h^{**}$ of the statement can be chosen to be $(K(l)+2)^{-1}\ h^*$ and the result follows.
\end{proof}

\section{Proof of Theorems \ref{th_almostsure} and \ref{th_unifalmostsure}}\label{app:2}

\subsection{Lower bound for the pointwise and local H\"older exponents}\label{sec:lowerHolder}

A lower bound for the local H\"older exponent is directly given by Corollary \ref{GaussKolm}. \\
For all $U_0\in\mathcal{A}$ and all 
$0<\alpha<\widetilde{\mathbb{\bbalpha}}_X(U_0)$, there exists $\rho_0>0$ and $K>0$ 
such that
\[
\forall U,V\in B_{\dA}(U_0,\rho_0);\quad
\EE\left[|X_U-X_V|^2\right] \leq K\ \dA(U, V)^{2\alpha}.
\]
Therefore, the sample paths of $X$ are almost surely $\nu$-H\"older continuous in $B_{\dA}(U_0,\rho_0)$ for all $\nu\in (0,\alpha)$, which leads to
$\alpha\leq\widetilde{\alphar}_X(U_0)$ almost surely. Then we get
\begin{equation*}
\PP\big( \widetilde{\alphar}_X(U_0)\geq \widetilde{\mathbb{\bbalpha}}_X(U_0) \big)=1.
\end{equation*}


By inequality (\ref{ineqholder}), any lower bound for the local H\"older exponent is also a lower bound for the pointwise exponent. 
Moreover it can be improved in the case of strict inequality
$0<\widetilde{\mathbb{\bbalpha}}_X(U_0) < \mathbb{\bbalpha}_X(U_0)$.

\noindent For any $\epsilon>0$, there exist $0<\rho_1<\rho_0$ and $M>0$ such that
\begin{equation*}
\forall\rho<\rho_1,\ \forall U, V\in B(U_0,\rho);\quad
\EE\left[\left|\frac{X_U-X_V}{\rho^{\mathbb{\bbalpha}_X(U_0)-\epsilon}}\right|^2\right]
\leq M\ \rho^{\epsilon}.
\end{equation*}
Then setting $\gamma=\mathbb{\bbalpha}_X(U_0)-\epsilon$, 
the exponential inequality for the centered Gaussian variable $X_U-X_V$ implies
\begin{equation*}
\PP\big( \left|X_U-X_V\right| \geq \rho^{\gamma} \big)
\leq \exp\left(-\frac{1}{2}\ \frac{\rho^{2\gamma}}{\EE[|X_U-X_V|^2]}\right)
\leq \exp\left( -\frac{1}{2} M \rho^{\epsilon} \right).
\end{equation*}


\noindent We consider the particular case $\rho=k_{n}^{-1/\qA}<\rho_1$ for $n\in\N$ large enough. Using the above estimate in the proof of Theorem \ref{kolmth} still leads to equation (\ref{eq:finChain}), where we had that on $\Omega^*$, for all $n\geq n^*$:
\begin{equation*}
\sup_{\substack{U,V\in \mathcal{D}\\ \dA(U,V)\leq \rho}} |X_U - X_V| \leq 3 \rho^{\gamma} \ .
\end{equation*}

\noindent Hence this inequality gives:
\begin{equation*}
\sup_{U,V\in B(U_0, k_{N}^{-1/\qA})} |X_U - X_V|\leq C\ k_{N}^{-\gamma/\qA} \quad\textrm{a.s.}
\end{equation*}
and since the sequence $\left(k_{n}^{-1/\qA} \right)_{n\in\N}$ is decreasing,
\begin{equation*}
\limsup_{\rho \rightarrow 0} \sup_{U,V\in B(U_0,\rho)} \frac{|X_U - X_V|}{\rho^{\gamma}} < \infty \quad\textrm{a.s.}
\end{equation*}

\noindent Therefore, $\forall \epsilon > 0$, $\alphar_X(U_0)\geq \mathbb{\bbalpha}_X(U_0) - \epsilon$ almost surely and
$\PP\left( \alphar_X(U_0)\geq \mathbb{\bbalpha}_X(U_0) \right)=1$.


\subsection{Upper bounds for the pointwise and local H\"older exponents}\label{sec:upperHolder}

As in \cite{herbin}, upper bounds for the pointwise and local H\"older exponents are given by the following two lemmas. Their proof are totally identical to the multiparameter setting.

\begin{lemma}\label{lem_uppoint}
Let $X=\left\{ X_U;\;U\in\mathcal{A} \right\}$ be a centered Gaussian process.
Assume that for $U_0\in\mathcal{A}$, there exists $\mu\in (0,1)$ such that for all $\epsilon>0$, there exist a sequence $\left(U_n\right)_{n\in\N^*}$ of $\mathcal{A}$ converging to $U_0$, and a constant $c>0$ such that
\begin{equation*}
\forall n\in\N^*;\quad
\EE\left[|X_{U_n}-X_{U_0}|^2\right] \geq c\; \dA(U_n,U_0)^{2\mu+\epsilon}.
\end{equation*}
Then, we have almost surely
\begin{equation*}
\alphar_X(U_0) \leq \mu.
\end{equation*}
\end{lemma}

Since the process $X$ has a finite deterministic H\"older exponent, for $\mu = \mathbb{\bbalpha}_X(U_0)$, one can find a sequence $(U_n)$ as in Lemma \ref{lem_uppoint}. 
Hence $\PP(\alphar_X(U_0) \leq \mathbb{\bbalpha}_X(U_0) ) = 1$.


\begin{lemma} \label{lem_upbound}
Let $X=\left\{ X_U;\;U\in\mathcal{A} \right\}$ be a centered Gaussian process.
Assume that for $U_0\in\mathcal{A}$, there exists $\mu\in (0,1)$ such that for all $\epsilon>0$, there exist two sequences $\left(U_n\right)_{n\in\N^*}$ and $\left(V_n\right)_{n\in\N^*}$ of $\mathcal{A}$ converging to $U_0$, and a constant $c>0$ such that
\begin{equation*}
\forall n\in\N^*;\quad
\EE\left[|X_{U_n}-X_{V_n}|^2\right] \geq c\; \dA(U_n,V_n)^{2\mu+\epsilon}.
\end{equation*}
Then, we have almost surely
\begin{equation*}
\widetilde{\alphar}_X(U_0) \leq \mu.
\end{equation*}
\end{lemma}

As for the pointwise case, $\PP(\widetilde{\alphar}_X(U_0) \leq \widetilde{\mathbb{\bbalpha}}_X(U_0) ) = 1$ follows from Lemma \ref{lem_upbound} with $\mu = \widetilde{\mathbb{\bbalpha}}_X(U_0)$.

\subsection{Proof of the uniform almost sure result}\label{sec:unif}

This section is devoted to the proof of Theorem~\ref{th_unifalmostsure}.
We only consider the local H\"older exponent. The uniform almost sure lower bound for the pointwise exponent is proved in a similar way. 

Starting with the lower bound, from Theorem \ref{kolmth}, for all $U_0 \in \Acal$ and all $\epsilon>0$, there is a modification $Y_{U_0}$ of $X$ which is $\alpha$-H\"older continuous for all $\alpha \in (0,\widetilde{\mathbb{\bbalpha}}_X(U_0)-\epsilon)$ on $B_{\dA}(U_0, \rho_0)$. 

\

In the first step, $\widetilde{\mathbb{\bbalpha}}_X$ is assumed to be constant over $\Acal$. Hence the local H\"older exponent of $Y_{U_0}$ satisfies almost surely
\begin{equation}\label{eq:uniflocholder}
\forall U \in B_{\dA}(U_0, \rho_0), \quad 
\widetilde{\alphar}_{Y_{U_0}}(U) \geq \widetilde{\mathbb{\bbalpha}}_X-\epsilon.
\end{equation}

\noindent The collection $\Acal$ is totally bounded, so it can be covered by a countable number of balls of radius at most $\eta$, for all $\eta>0$. Let $B$ be one of these balls. For all $U_0 \in \Acal$, we consider $\rho_0>0$ such that (\ref{eq:uniflocholder}) holds. We have obviously
\begin{equation*}
B \subseteq \bigcup_{U_0 \in B} B_{\dA}(U_0, \rho_0).
\end{equation*}
For each open ball, there exists an integer $n$ such that $B_{\dA}(U_0, \rho_0) \cap \Acal_n \neq \emptyset$ so that for $V_0\in B_{\dA}(U_0, \rho_0)\cap\mathcal{A}_n$, there exists an integer $m_0$ such that $U_0\in B_{\dA}(V_0, 2^{-m_0}) \subseteq B_{\dA}(U_0, \rho_0)$. Thus $B \subseteq \bigcup B_{\dA}(V_0, 2^{-m_0})$, where the union is countable. 
Each of these balls satisfies
\begin{equation*}
\PP\left( \forall U \in B_{\dA}(V_0, 2^{-m_0}),\ \widetilde{\alphar}_X(U) \geq \widetilde{\mathbb{\bbalpha}}_X-\epsilon \right) = 1,
\end{equation*}
and since $\mathcal{A}$ is a countable union of balls $B_{\dA}(V_0, 2^{-m_0})$, we get
\begin{equation*}
\PP\left( \forall U \in \Acal,\ \widetilde{\alphar}_X(U)\geq \widetilde{\mathbb{\bbalpha}}_X - \epsilon \right) = 1.
\end{equation*}
Taking $\epsilon\in\Q_+^*$, we conclude that
\begin{equation}\label{eq_111}
\PP\left( \forall U \in \Acal,\ \widetilde{\alphar}_X(U)\geq\widetilde{\mathbb{\bbalpha}}_X \right) = 1.
\end{equation}

\

In the general case of a not constant exponent $\widetilde{\mathbb{\bbalpha}}_X$, for any ball $B$ of radius $\eta$ previously introduced, we set $\beta = \inf_{U\in B} \widetilde{\mathbb{\bbalpha}}_X(U) - \epsilon$, $\epsilon>0$. Then, there exists a constant $C>0$ such that
\begin{equation*}
\forall U,V \in B,\quad \EE[|X_U - X_V|^2] \leq C\ \dA(U,V)^{2\beta}.
\end{equation*}
In a similar way as we proved (\ref{eq_111}), we deduce the existence of an event $\Omega^* \subseteq \Omega$ of probability one such that for all $\omega \in \Omega^*$:
\begin{align*}
\forall U \in \mathcal{A},\ &\forall n\geq 0,\ \forall\epsilon\in\Q_+^*, \\
&\forall U_0 \in B_{\dA}(U,2^{-n}),\quad
\widetilde{\alphar}_X(U_0) \geq \inf_{V\in B_{\dA}(U,2^{-n})} \widetilde{\mathbb{\bbalpha}}_X(V) - \epsilon.
\end{align*}
By letting $n\rightarrow\infty$, the previous equation leads to
\begin{equation*}
\PP\left( \forall U_0 \in \Acal,\quad
\widetilde{\alphar}_X(U_0) \geq 
\liminf_{U\rightarrow U_0} \widetilde{\mathbb{\bbalpha}}_X(U) \right) = 1.
\end{equation*}

\

In order to prove the converse inequality (which holds only for the local exponent), we adapt a proof in \cite{ehjlv}. We assume that $\widetilde{\mathbb{\bbalpha}}_X$ is constant on $\mathcal{A}$, the case where it is not being similar to the lower bound. Using the fact that $\mathcal{D} = \bigcup_{n\in\N} \Acal_n$ is countable, Lemma \ref{lem_upbound} gives that $\PP(\forall U \in\mathcal{D},\ \widetilde{\alphar}_X(U)\leq \widetilde{\mathbb{\bbalpha}}_X) = 1$.
Let $\Omega'\in\mathcal{F}$ be the set of $\omega$, such that $\widetilde{\alphar}_X(U)\leq\widetilde{\mathbb{\bbalpha}}_X$ for all $U\in\mathcal{D}$. 
Let $U_0 \in \Acal \setminus \mathcal{D}$. Let $(U^{(i)})_{i\in\N}$ be a sequence in $\mathcal{D}$ converging to $U_0$. 
On $\Omega'$, $\widetilde{\alphar}_X(U^{(i)})\leq \widetilde{\mathbb{\bbalpha}}_X$, for all $i\in\N$. 
For each fixed $i\in\N$, there exist two sequences $(V^{(i)}_n)_{n\in\N}$ and $(W^{(i)}_n)_{n\in\N}$ in $\Acal$ converging to $U^{(i)}$ as $n\rightarrow\infty$, and for all $n\in\N$,
\begin{equation*}
\lim_{n\rightarrow + \infty} \frac{|X_{V^{(i)}_n} - X_{W^{(i)}_n}|}{\dA(V^{(i)}_n,W^{(i)}_n)^{\widetilde{\mathbb{\bbalpha}}_X+\epsilon}} = +\infty.
\end{equation*}

\noindent As in \cite{ehjlv}, we build two other sequences $(V_n)_{n\in\N}$ and $(W_n)_{n\in\N}$ such that $V_n\rightarrow U_0$ and $W_n \rightarrow U_0$, and the following equality concludes the proof:
\begin{equation*}
\lim_{n\rightarrow + \infty} \frac{|X_{V_n} - X_{W_n}|}{\dA(V_n,W_n)^{\widetilde{\mathbb{\bbalpha}}_X+\epsilon}} = +\infty.
\end{equation*}

\end{appendices}

\vspace{0.5cm}



\end{document}